\documentclass{amsart}
\usepackage{stmaryrd}
\usepackage[all]{xy}
\usepackage{mathrsfs}
\usepackage{amscd,amssymb,color, amsmath}
\usepackage{todonotes}
\usepackage{hyperref}
\usepackage{url}

\let\sma\wedge
\newcommand{\htp}{\simeq}
\renewcommand{\to}{\mathchoice{\longrightarrow}{\rightarrow}{\rightarrow}{\rightarrow}}
\newcommand{\from}{\mathchoice{\longleftarrow}{\leftarrow}{\leftarrow}{\leftarrow}}

\newcommand{\cC}{{\mathcal C}}
\newcommand{\cD}{{\mathcal D}}
\newcommand{\cE}{{\mathcal E}}
\newcommand{\cF}{{\mathcal F}}
\newcommand{\cI}{{\mathcal I}}
\newcommand{\cK}{{\mathcal K}}
\newcommand{\cL}{{\mathcal L}}

\newcommand{\cN}{{\mathcal N}}
\newcommand{\cO}{{\mathcal O}}

\let\catsymbfont\mathcal

\newcommand{\aC}{{\catsymbfont{C}}}
\newcommand{\aD}{{\catsymbfont{D}}}
\newcommand{\aE}{{\catsymbfont{E}}}

\newcommand{\aO}{{\catsymbfont{O}}}

\newcommand{\aT}{{\catsymbfont{T}}}

\newcommand{\D}{\mathcal{D}}
\newcommand{\E}{\mathcal{E}}

\newcommand{\K}{\mathcal{K}}

\newcommand{\bO}{\mathbb{O}}
\newcommand{\bP}{{\mathbb{P}}}

\newcommand{\bR}{{\mathbb{R}}}

\newcommand{\bZ}{{\mathbb{Z}}}

\mathchardef\endash="2D

\newcommand{\CoInd}{\textnormal{CoInd}}

\newcommand{\Sp}{\mathcal Sp}
\newcommand{\Emb}{\textnormal{Emb}}
\newcommand{\Set}{\mathcal Set}

\newcommand{\Ninfty}{N_\infty}
\newcommand{\Coef}{\mathcal Coef}

\newcommand{\Top}{\mathcal Top}
\newcommand{\Ab}{\mathcal Ab}

\newcommand{\mC}{{\underline{\cC}}}
\newcommand{\mSet}{\underline{\Set}}
\newcommand{\mM}{\underline{M}}
\newcommand{\mF}{\underline{\cF}}

\newcommand{\Coll}{\mathrm{SymSeq}}

\def\quickop#1{\expandafter\DeclareMathOperator\csname
#1\endcsname{#1}}
\quickop{id}\quickop{Id}\quickop{colim}\quickop{op}
\quickop{co}\quickop{Ar}\quickop{sing}\quickop{Hom}\quickop{w}\quickop{Ho}
\quickop{ob}\quickop{diag}\quickop{Stab}\quickop{Cat}\quickop{Mot}\quickop{Mod}
\quickop{map}\quickop{Cone}\quickop{End}\quickop{Idem}\quickop{Perf}\quickop{Ind}
\quickop{Gap}\quickop{Shv}\quickop{Spaces}\quickop{cof}\quickop{ex}\quickop{perf}
\quickop{tri}
\quickop{Fib}\quickop{Nat}\quickop{haut}\quickop{holim}\quickop{Tor}\quickop{Ext}\quickop{Diff}\quickop{lax}\quickop{Map}\quickop{Comm}\quickop{Tr}\quickop{Res}\quickop{conn}\quickop{cyc}\quickop{Sing}

\newcommand{\poset}{\cI}

\newcommand{\Nop}{\cN_{\infty}\text{-}\mbox{Op}}
\newcommand{\Gop}{G\text{-}{\mbox{Op}}}

\newcommand{\sset}{\textrm{Set}^{\Delta^{\op}}}
\newcommand{\gsset}{\textrm{GSet}^{\Delta^{\op}}}

\numberwithin{equation}{section}

\newtheorem{theorem}[equation]{Theorem}
\newtheorem*{theorem*}{Theorem}
\newtheorem{corollary}[equation]{Corollary}
\newtheorem{lemma}[equation]{Lemma}
\newtheorem{proposition}[equation]{Proposition}
\newtheorem{conjecture}[equation]{Conjecture}
\theoremstyle{definition}
\newtheorem{definition}[equation]{Definition}
\newtheorem{example}[equation]{Example}
\newtheorem*{motivatingexample}{Motivating Example}
\theoremstyle{remark}
\newtheorem{remark}[equation]{Remark}

\newtheorem{construction}[equation]{Construction}

\newtheorem{warning}[equation]{Warning}

\xyoption{arrow}
\xyoption{curve}
\xyoption{matrix}
\xyoption{cmtip}
\SelectTips{cm}{}
\CompileMatrices

\newcommand{\Image}{\textrm{Im}}

\newdir{ >}{{}*!/-5pt/\dir{>}}

\begin{document}

\title[Equivariant Operadic Multiplications]{Operadic multiplications in equivariant spectra, norms, and transfers}

\author[A.J. Blumberg]{Andrew J. Blumberg}
\address{Department of Mathematics, University of Texas,
Austin, TX \ 78712}
\email{blumberg@math.utexas.edu}
\thanks{A.~J.~Blumberg was supported in part by NSF grants DMS-0906105
and DMS-1151577}

\author[M.A.Hill]{Michael A.~Hill}
\address{University of Virginia \\ Charlottesville, VA 22904}
\email{mikehill@virginia.edu}
\thanks{M.~A.~Hill was supported in part by NSF DMS--1207774, the Sloan Foundation, and by DARPA through the Air Force Office of Scientific Research (AFOSR) grant number HR0011-10-1-0054}

\begin{abstract}
We study homotopy-coherent commutative multiplicative structures on
equivariant spaces and spectra.  We define $\Ninfty$ operads,
equivariant generalizations of $E_{\infty}$ operads.  Algebras in
equivariant spectra over an $\Ninfty$ operad model homotopically
commutative equivariant ring spectra that only admit certain
collections of Hill-Hopkins-Ravenel norms, determined by the operad.
Analogously, algebras in equivariant spaces over an $\Ninfty$ operad
provide explicit constructions of certain transfers.  This
characterization yields a conceptual explanation of the structure of
equivariant infinite loop spaces.

To explain the relationship between norms, transfers, and $\Ninfty$
operads, we discuss the general features of these operads, linking
their properties to families of finite sets with group actions and
analyzing their behavior under norms and geometric fixed points.  A
surprising consequence of our study is that in stark contract to the
classical setting, equivariantly the little disks and linear
isometries operads for a general incomplete universe $U$ need not
determine the same algebras.

Our work is motivated by the need to provide a framework to describe
the flavors of commutativity seen in recent work of the second author
and Hopkins on localization of equivariant commutative ring spectra.
\end{abstract}

\maketitle

\setcounter{tocdepth}{1}
\tableofcontents

\section{Introduction}

One of the most important ideas in modern stable homotopy theory is
the notion of a structured ring spectrum, an enhancement of the
representing object for a multiplicative cohomology theory.  A
structured ring spectrum is a spectrum equipped with a
homotopy-coherent multiplication; classically the coherence data is
packaged up in an operad.  When the multiplication is coherently
commutative (as in the familiar examples of $H\mathbb{Z}$, $ku$, and
$MU$), the classical operadic description of the multiplication
involves an $E_\infty$ operad.

May originally observed that all $E_\infty$ operads are equivalent up
to a zig-zag of maps of operads~\cite{Maygils} and showed that
equivalent $E_\infty$ operads have equivalent homotopical categories
of algebras.  As an elaboration of this basic insight it is now
well-understood that all possible notions of commutative ring spectrum
agree.  For instance, in the symmetric monoidal categories of EKMM
$S$-modules~\cite{EKMM} and of diagram spectra~\cite{MMSS} (i.e.,
symmetric spectra and orthogonal spectra), the associated categories
of commutative monoids are homotopically equivalent to the classical
category of $E_\infty$-ring spectra~\cite{MayQuinnRay, LMS}.
Moreover, the homotopy theories of the categories of commutative
monoids are equivalent to the homotopy theories of the category of
algebras over any reasonable $E_\infty$ operad~\cite[\S II.4]{EKMM}.

Our focus in this paper is on equivariant generalizations of
$E_\infty$ ring spectra.  At first blush, it might seem that we can
give an analogous account of the situation.  After all, for any
compact Lie group $G$ and universe $U$ of finite dimensional
$G$-representations, there is the classical notion of an equivariant
$E_\infty$ ring spectrum structured by the equivariant linear
isometries operad on $U$~\cite{LMS}.  For each $U$, there are
equivariant analogues of the modern categories of spectra (i.e.,
equivariant orthogonal spectra and equivariant $S$-modules) that are
symmetric monoidal categories \cite{MM, HHR}.  Moreover, once again
commutative monoids in these categories are equivalent to classical
equivariant $E_\infty$ ring spectra (see~\cite[\S4-5]{MM}).

However, this is not the whole story.  Fix a symmetric monoidal
category $\Sp_G$ of equivariant spectra that is tensored over
$G$-spaces and is a model of the equivariant stable homotopy category
specified by a complete universe $U$.  For any operad $\aO$ of
$G$-spaces, we can form the category of $\aO$-algebras in $\Sp_G$.
There are many different $G$-operads $\aO$ such that the underlying
non-equivariant operad is $E_\infty$; for instance, for any universe
$U'$, the equivariant linear isometries operad over $U'$ provides an
example.  Any operad with that property might be entitled to be
thought of as a $G$-$E_\infty$ operad.  However, operadic algebras in
$\Sp_G$ over different such operads can look very different, as the
following example illustrates.
 
\begin{motivatingexample}
Let $\cE$ be an $E_{\infty}$ operad in spaces, and view it as an
operad in $G$-spaces by giving it the trivial $G$-action. Thus the
$n$\textsuperscript{th} space is equivalent to $E\Sigma_{n}$ with a
trivial $G$-action. Let $\cE_{G}$ denote any $E_{\infty}$ operad in
$G$-spaces for which the $n$\textsuperscript{th} space $(\cE_G)_n$ is
a universal space for $(G\times\Sigma_{n})$-bundles in $G$-spaces
(e.g., the $G$-linear isometries operad for a complete universe $U$).
Then algebras over $\cE$ and algebras over $\cE_{G}$ in orthogonal
$G$-spectra are different.  In fact, for almost all positive cofibrant
orthogonal $G$-spectra $E$,
\[
\cE_{n+}\wedge_{\Sigma_{n}}E^{\wedge n}\not\simeq (\cE_G)_{n+}\wedge_{\Sigma_{n}}E^{\wedge n}.
\]
The easiest way to see this generic inequality is by computing the
$G$-geometric fixed points. If $E=\Sigma^{\infty}G_{+}$, then for all
$n$, $E^{\wedge n}$ is a free $G$-spectrum. This means, in particular,
that the geometric fixed points of the free $\cE$-algebra on $E$ are
$S^{0}$. However, if $n=|G|$, then $(\cE_{G})_n$ has cells of the form
$G\times\Sigma_{n}/\Gamma$, where $\Gamma$ is the graph of the
homomorphism $G\to\Sigma_{n}$ describing the left action of $G$ on
itself. The $G$-spectrum
\[
(G\times\Sigma_{n}/\Gamma)_{+}\wedge_{\Sigma_{n}} E^{\wedge n}
\]
is the Hill-Hopkins-Ravenel norm $N_{e}^{G}(E)$, and in particular,
the geometric fixed points are non-trivial.
\end{motivatingexample}

Moreover, it turns out there are many intermediate classes of
$G$-operads that structure equivariant commutative ring spectra that
are richer than $\aE$-algebras but are not $\aE_G$-algebras.  Our
interest in these different notions of equivariant commutative ring
spectra was motivated by recent work of Hopkins and the second author
which showed that equivariantly, Bousfield localization does not
necessarily take $\aE_G$-algebras to $\aE_G$-algebras.  For formal
reasons, the Bousfield localization of any equivariant commutative
ring spectrum must have a multiplication that is an $\aE$-algebra, but
that is all that is guaranteed.  An antecedent of this general result
appears in work of McClure~\cite{mccluretate} which shows that the
Tate spectrum of an $\aE_G$-algebra only necessarily has a
multiplication that is structured by $\aE$ and is usually not itself
an $\aE_G$-algebra.

Our goal in this paper is to provide conceptual descriptions of these
intermediate multiplications on equivariant spaces and spectra in
terms of the Hill-Hopkins-Ravenel norm.  We do this via a careful
study of the $G$-operads that structure intermediate multiplications,
which we characterize in terms of the allowable norms on algebras over
them, as suggested by the example above.  For this reason, we refer to
such operads as $\Ninfty$ operads.

Fix a finite group $G$.  A $G$-operad $\cO$ consists of a sequence of
$G \times \Sigma_{n}$ spaces $\cO_{n}$, $n \geq 0$, equipped with a
$G$-fixed identity element $1 \in \cO_{1}$ and a composition map
satisfying equivariant analogues of the usual axioms (see
Definition~\ref{def:goper} for details).

\begin{definition}\label{def:introgeinfop}
An $\Ninfty$ operad is a $G$-operad such that
\begin{enumerate}
\item The space $\cO_{0}$ is $G$-contractible, 
\item The action of $\Sigma_{n}$ on $\cO_{n}$ is free, and
\item $\cO_{n}$ is a universal space for a family $\cF_{n}(\cO)$
of subgroups of $G\times\Sigma_{n}$ which contains all subgroups of
the form $H\times\{1\}$.
\end{enumerate}
In particular, the space $\cO_{1}$ is also $G$-contractible.
\end{definition}

Forgetting the $G$-action, an $\Ninfty$ operad yields a
non-equivariant $E_\infty$ operad.  Examples include the equivariant
little isometries operads and equivariant little disks operads; see
Definition~\ref{def:operexa} for details.

Our first main theorem is a classification of $\Ninfty$ operads in
terms of the relationship between the universal spaces $\cO_{n}$
forced by the operadic structure maps.  Associated to an $\Ninfty$
operad, there is a naturally defined collection (indexed by the
subgroups of $G$) of categories of finite sets, called admissible
sets.  We can organize the admissible sets as follows.  Define a
symmetric monoidal coefficient system to be a contravariant functor
$\underline{\cC}$ from the orbit category of $G$ to the category of
symmetric monoidal categories and strong monoidal functors.

There is a canonical coefficient system that assigns to the orbit
$G/H$ the category of finite $H$-sets and $H$-maps, with symmetric
monoidal product given by disjoint union.  We have a poset $\poset$ of
certain sub-coefficient systems of the canonical coefficient system,
ordered by inclusion (i.e., the ones closed under Cartesian product
and induction, see Definition~\ref{def:poset}).  Let $\Nop$ denote the
category of $\Ninfty$ operads, regarded as a full subcategory of
$G$-operads and $G$-operad maps.

\begin{theorem}
There is a functor
\[
\mC\colon \Nop \to \poset
\]
which descends to a fully-faithful embedding
\[
\mC\colon \Ho(\Nop)\to\poset,
\]
where the homotopy category is formed with respect to the maps of
$G$-operads which are levelwise $G \times \Sigma_n$-equivalences.
\end{theorem}

We conjecture that in fact this embedding is an equivalence of
categories; as we explain in Section~\ref{sec:fullness}, there are
natural candidate $\Ninfty$ operads to represent each object in
$\poset$.  An interesting question is to determine if all homotopy
types are realized by equivariant little disks or linear isometries
operads.

\begin{remark}
The proof of the preceding theorem involves a calculation of the
derived mapping space between $\Ninfty$ operads (see
Proposition~\ref{prop:mapcontract}); in particular, we show that the
space of endomorphisms of an $\Ninfty$ operad is contractible.
\end{remark}
 
The import of this classification theorem is that it establishes that
$\Ninfty$ operads are essentially completely controlled by the
isotropy condition in the definition.  This allows for very surprising
results about the cofree spectra with an action of an $\Ninfty$
operad.

\begin{theorem}
If $\cO$ is an $\Ninfty$ operad and $R$ is an $\cO$-algebra with the
property that the natural map 
\[
R\to F(EG_{+},R)
\]
is an equivariant equivalence, then $R$ is equivalent (as
$\cO$-algebras) to an $\cE_G$-algebra.
\end{theorem}

Our other main theorems provide
a characterization of structures on algebras over $\Ninfty$ operads.
The indexed product construction that underlies the norm makes sense
in the symmetric monoidal category of $G$-spaces with the Cartesian
product, where the resulting functor is simply coinduction.  In this
situation, we show in Theorem~\ref{thm:GeneralTransfers} that an
algebra over an $\Ninfty$ operad has precisely those transfers $H \to
G$ such that $G/H$ is an admissible $G$-set.  Specifically, we have
the following result.

\begin{theorem}
For an algebra $X$ in $G$-spaces over a suitable $N_\infty$ operad,
the abelian group valued coefficient system
\[
\underline{\pi_{k}(X)}\colon\mSet\to\Ab
\]
defined by
\[
(T\in\Set^{H})\mapsto \pi_{k}\big(F(T,X)^{H}\big)
\]
has transfer maps
\[
f_{\ast}\colon\pi_{k}\big(F(T,i_{H}^{\ast}X)^{H}\big)\to \pi_{k}\big(F(S,i_{H}^{\ast}X)^{H}\big)
\]
for any $H$-map $f \colon T \to S$ of admissible $H$-sets and all
$k\geq 0$.  Moreover, for the little disks and Steiner operads, these
transfers maps agree with the classical transfers.
\end{theorem}

These are therefore incomplete Mackey functors, studied by Lewis
during his analysis of incomplete universes \cite{LewisHurewicz,
LewisChange}.

\begin{remark}
In the result above, ``suitable'' refers to a certain technical
property of $\Ninfty$ operads that we prove for the equivariant
Steiner and linear isometries operads in Section~\ref{sec:inter}.
\end{remark}

In orthogonal $G$-spectra, we show in
Theorem~\ref{thm:ExistenceofNorms} that an algebra over a suitable
$\Ninfty$ operad is characterized as a $G$-spectrum equipped with maps
\[
G_{+}\wedge_{H}N^{T} \iota^*_H R\to R
\]
for the admissible $H$-sets $T$.  This gives rise to the following
characterization:

\begin{theorem}
If $R$ is an algebra in orthogonal $G$-spectra over an $\Ninfty$
operad $\cO$, then 
\[
\underline{\pi_{0}(R)}
\]
is a commutative Green functor.
 
If the $\cO$ action interchanges with itself, then for any admissible
$H$-set $H/K$ we have a ``norm map''
\[
\underline{\pi_{0}(R)}(G/K)\xrightarrow{n_{K}^{H}}\underline{\pi_{0}(R)}(G/H)
\]
which is a homomorphism of commutative multiplicative monoids.

The maps $n_{K}^{H}$ satisfy the multiplicative version of the Mackey
double-coset formula.
\end{theorem}

Thus just as the homotopy groups of algebras in spaces over the
Steiner operad on an incomplete universe gave incomplete Mackey
functors with only some transfers, the zeroth homotopy group of an
algebra in spectra over the linear isometries operad on an incomplete
universe gives incomplete Tambara functors with only some norms.

The paper is organized as follows.  In Section~\ref{sec:conv}, we
explain our assumptions and conventions about the kinds of operadic
actions and categories of equivariant spectra we are working with.  We
introduce the notion of $\Ninfty$ operads in Section~\ref{sec:defs}.
We use this to explain in Section~\ref{sec:Admissibles} that
associated to an $\Ninfty$ operad, there is a naturally defined
collection (indexed by the subgroups of $G$) of categories of finite
sets, called admissible sets, and that if two operads have the same
admissible sets, then they are equivalent.  In Section~\ref{sec:LUDU},
we perform a surprising computation: we show that for a generic
incomplete universe, the little disks operad and the linear isometries
operad are different.  In Section~\ref{sec:hocat} we discuss the
connection between the homotopy category of $\Ninfty$ operads and the
poset $\poset$.  In Section~\ref{sec:Algebras}, we then show that the
admissible sets correspond to indexed products that an algebra over
the operad must have.  In Section~\ref{sec:TransfersAndNorms}, we work
out this characterization in equivariant spaces and spectra.  In the
case of algebras in $G$-spaces over $\Ninfty$ operads, this
perspective explains the transfers that arise in $G$-equivariant
infinite loop space theory.  In the case of equivariant spectra, this
structure controls which norms occur in a ring spectrum.  Finally, in
the appendix we collect some miscellaneous technical results: in
Section~\ref{sec:monadicalgebras} we show that weakly equivalent
$\Ninfty$ operads have equivalent homotopical categories of algebras
and we explain the comparison to rigid realizations of $\Ninfty$
operadic algebras in terms of equivariant EKMM spectra, and finally in
Section~\ref{sec:appfix} we describe geometric fixed points of
algebras.

\subsection{Acknowledgements}

We want to thank Mike Hopkins, Mike Mandell, and Peter May for much
useful advice and many helpful conversations.  The paper benefited
enormously from a careful reading of a previous draft by Peter May.
The paper was also improved by comments from Aaron Royer, Anna-Marie
Bohmann, Emanuele Dotto, Justin Noel, Tomer Schlank, and David White.

\section{Conventions on operadic algebras in equivariant spectra}\label{sec:conv}
 
Fix a finite group $G$ and a complete universe $U$ of
$G$-representations.  Let $\Sp_{G}$ denote the category of orthogonal
$G$-spectra~\cite{MM}.  We will always regard $\Sp_G$ as equipped with
the homotopy theory specified by the weak equivalences detected by the
equivariant stable homotopy groups indexed by $U$~\cite[III.3.2]{MM};
$\Sp_G$ is a model of the equivariant stable category and all
representation spheres are invertible~\cite[III.3.8]{MM}.  However,
the multiplicative structures we study are often described by linear
isometries operads over other universes and in general the language of
incomplete universes is very useful in describing $\Ninfty$ operads.
The key point we want to emphasize is that although the multiplicative
structure varies, the additive structure does not.

We now want to be clear about what we mean by an operadic algebra in
$\Sp_{G}$.  Since $\Sp_{G}$ is tensored over $G$-spaces (with the
tensor of a $G$-space $A$ and an orthogonal $G$-spectrum $E$ computed
as $A_+ \sma E$), we can define the category $\Sp_{G}[\aO]$ of
$\aO$-algebras for any operad $\aO$ in $G$-spaces.  This is the notion
of operadic algebra we study in this paper.  However, there is
the potential for terminological confusion: even when $\aO$ is a
classical $G$-$E_\infty$ operad, for instance the $G$-linear
isometries operad, the category $\Sp_{G}[\aO]$ {\em is not} equivalent
to the classical category of $G$-$E_\infty$ ring spectra~\cite{LMS}.
The latter is defined using the category of ``coordinate-free''
$G$-spectra and the twisted half-smash product, and requires of
necessity operads augmented over the $G$-linear isometries operad.
(This terminological point is clearly explained
in~\cite[\S13]{Mayrant}.)  Note that it is the case that the homotopy
categories of $\Sp_{G}[\aO]$ and the classical category of
$G$-$E_\infty$ ring spectra are equivalent.
See Appendix~\ref{sec:monadicalgebras} for further discussion of such
comparison results.

We could also have worked with the equivariant analogues of EKMM
$S$-modules (e.g., see~\cite{ElmendorfMay} for a discussion of this
category) based on $U$.  However, since we rely at various points on
the homotopical analysis of the norm from~\cite[App.~B]{HHR}, it is
convenient for our purposes to work with orthogonal $G$-spectra.  We
have no doubt that our theorems are independent of the specific model
of the equivariant stable category, however.

Finally, we note that our results have analogues in the situation when
the (additive) homotopy theory on $\Sp_G$ is indexed on an incomplete
universe.  However, in this situation some care must be taken. The
underlying analysis was begun by Lewis~\cite{lewisincomplete}, who
analyzed the homotopy theory of $G$-spectra on incomplete
universes, and various subtleties about the connections between the
additive and multiplicative structures are known to experts.  We leave
the elaboration in this setting to the interested reader.  However, we
note that our analysis in Section~\ref{sec:LUDU} of the linear
isometries operads also provides a criterion for the special case when
both the additive and multiplicative universes are the same but
potentially incomplete.

\section{Equivariant operads and indexing systems}\label{sec:defs}

In this section, we define $\Ninfty$ operads and give a number of
examples.  We then move on to introduce definitions and notations for
{\em indexing systems}, which allows us to precisely state our main
result describing the homotopy category of $\Ninfty$ operads in terms
of a certain poset.

\subsection{Equivariant $\Ninfty$ operads}\label{ssec:Ninfty}

In this section we review the definitions and standard examples of
$G$-operads that we will work with.

\begin{definition}\label{def:goper}
A $G$-operad $\cO$ consists of a sequence of $G \times \Sigma_{n}$
spaces $\cO_{n}$, $n \geq 0$, such that
\begin{enumerate}
\item There is a $G$-fixed identity
element $1 \in \cO_{1}$,
\item and we have $G$-equivariant compositions maps 
\[
\cO_{k}\times\cO_{n_{1}}\times\dots\times\cO_{n_{k}}\to\cO_{n_{1}+\dots+n_{k}}
\]
which satisfy the usual compatibility conditions with each other and
with the action of the symmetric groups
(see~\cite[2.1]{costenoblewaner}). In particular, if
$n_{1}=\dots=n_{k}=n$, then the map is actually
$(G\times \Sigma_{k}\wr\Sigma_{n})$-equivariant.
\end{enumerate}
When $\cO_0 = *$, we say that $\cO$ is a reduced operad.
\end{definition}

\begin{remark}
Note that in contrast to the usual convention, we will treat
$G$-operads as having left actions of symmetric groups via the
inversion, as this makes certain formulas easier to understand. It
also allows a simultaneous equivariant treatment of the $G$ and
$\Sigma_n$-actions.
\end{remark}

We will primarily be interested in the equivariant analogues of
$E_\infty$ operads. For this, we need the notion of a family and of a
universal space for a family.

\begin{definition}\label{def:Family}
A family for a group $G$ is a collection of subgroups closed under
passage to subgroup and under conjugacy.

If $\mathcal F$ is a family, then a universal space for $\mathcal F$
is a $G$-space $E\mathcal F$ such that for all subgroups $H$,
\[
(E\mathcal F)^H\simeq \begin{cases}
\ast & H\in\mathcal F, \\
\varnothing & H\not\in\mathcal F.
\end{cases}
\]
\end{definition}

For later purposes, there is an equivalent definition that is more
categorical.

\begin{definition}
A sieve in a category $\cC$ is a full subcategory $\mathcal D$ such
that if $B\to C$ is in $\mathcal D$ and if $A\to B$ is in $\cC$, then
the composite $A\to C$ is in $\mathcal D$.
\end{definition}

With this, we have two equivalent formulations of a family.

\begin{proposition}\label{prop:Families}\mbox{}

\begin{enumerate}
\item A family of subgroups $\cF$ determines a sieve in the orbit
category by considering the full subcategory generated by the objects
$G/H$ for $H\in\cF$. Similarly, the collection of all $H$ such that
$G/H$ is in a sieve in $\cO^{G}$ forms a family. 

\item A family of subgroups $\cF$ is also equivalent to a sieve
  $\Set_{\cF}$ in the category of finite $G$-sets, where again the
  identification specifies that $T$ is in the sieve if and only if the
  stabilizers of points of $T$ are in the family.
\end{enumerate}
\end{proposition}

\begin{remark}
An equivalent condition to condition (ii) in
Proposition~\ref{prop:Families} is that the sieve in $G$-sets is the
full subcategory generated by those $G$-sets $T$ such that the space
of $G$-equivariant maps from $T$ to $E\cF$ is contractible.
\end{remark}

\begin{definition}\label{def:geinfop}
An $\Ninfty$ operad is a $G$-operad such that
\begin{enumerate}
\item The space $\cO_{0}$ is $G$-contractible, 
\item The action of $\Sigma_{n}$ on $\cO_{n}$ is free,
\item and $\cO_{n}$ is a universal space for a family $\cF_{n}(\cO)$
of subgroups of $G\times\Sigma_{n}$ which contains all subgroups of
the form $H\times\{1\}$.
\end{enumerate}
In particular, the space $\cO_{1}$ is also $G$-contractible.
\end{definition}

Historically, most sources have focused on the situation where
$\cO_{n}$ is a universal principle $(G,\Sigma_n)$-bundle; i.e.,
$\cO_{n}^{\Lambda}$ is nonempty and contractible for $\Lambda$ which intersects $\Sigma_{n}$ trivially
(e.g., see~\cite{costenoblewaner}).  As we shall recall, this is the
analogue of restricting attention to a complete universe.  We will
refer to such an $\Ninfty$ operad as ``complete'' and follow the
literature in calling these $E_{\infty}$ $G$-operads.  For any
$H \subset G$, there is a forgetful functor from $\Ninfty$ operads on
$G$ to $\Ninfty$ operads on $H$.  When $G = {e}$, it is clear from the
definition that an $\Ninfty$ operad is an ordinary $E_{\infty}$
operad.

\begin{lemma}
The underlying non-equivariant operad for any $\Ninfty$ operad is an
$E_\infty$ operad.
\end{lemma}

The category $\Nop$ of $\Ninfty$ operads, regarded as a full
subcategory of the category of $G$-operads and $G$-operad maps, is a
category with weak equivalences.  The weak equivalences are ultimately
lifted from the homotopy theory on $G$-spaces where a map $f \colon
X \to Y$ of $G$-spaces is a $G$-equivalence if the induced maps
$f^H \colon X^H \to Y^H$ on $H$-fixed points are nonequivariant weak
equivalences for each (closed) subgroup $H \subset G$.

\begin{definition}\label{defn:weak}
A map of $\cO \to \cO'$ of $G$-operads is a weak equivalence if each
map $\cO(n)^{\Gamma} \to \cO'(n)^{\Gamma}$ is an equivalence for all subgroups $\Gamma
\subseteq G \times \Sigma_n$.
\end{definition}

Note that this definition of weak equivalence does not generalize the
usual weak equivalences on operads (i.e., the maps of operads which
are underlying equivalences of spaces for each $n$) when $G = e$;
rather, this is a generalization of Rezk's
notion of weak equivalence of operads~\cite[\S 3.2.10]{Rezk}.  The
generalization of the usual notion would lead to a weak equivalence of
$\Ninfty$ operads being a levelwise $G$-equivalence of spaces, and
under this definition the linear isometries operad on a genuine
universe and any $G$-trivial $E_\infty$ operad would be equivalent via
a zig-zag.

\begin{remark}
One can also ask for a weaker notion of weak equivalence wherein one checks only the fixed points for subgroups of $G\times\Sigma_{n}$ which intersect $\Sigma_{n}$ trivially. This arises for instance in work of Dotto and Schlank considering $G$-operads in terms of presheaves on certain subcategories of the orbit category. For $\Ninfty$ operads, the two notions coincide, since all of the other fixed points are assumed to be empty; for this reason, we do not discuss this further.
\end{remark}

We now turn to examples. The $\Ninfty$ operads which arise most
frequently in equivariant algebraic topology are the linear isometries
operad on a universe $U$ and variants of the little disks operad on a
universe $U$.  To be precise, let $U$ denote a countably
infinite-dimensional real $G$-inner product space which contains each
finite dimensional sub-representation infinitely often and for which
the $G$-fixed points are non-empty.  We emphasize that $U$ is not
assumed to be complete.  Our presentation is heavily based on the
excellent treatment of~\cite[\S 10]{guilloumay}; we refer the
interested reader to that paper for more discussion.

\begin{definition}{\mbox{}}\label{def:operexa}
\begin{enumerate}

\item The {\emph{linear isometries}} operad $ \cL(U)$ has
$n$\textsuperscript{th} space $ \cL(U^{n},U)$ of (nonequivariant)
linear isometries from $U^{n}$ to $U$.  The $G\times\Sigma_{n}$-action
is by conjugation and the diagonal action.  The distinguished element
$1 \in \cL(U,U)$ is the identity map, and the structure maps are
induced from composition.
 
\item The {\emph{little disks}} operad $\D(U)$ has
$n$\textsuperscript{th} space $\D(U)_n$ given as the colimit of
embeddings of $n$ copies of the disk in the unit disk of a finite
subrepresentation $V$ in $U$.  Precisely, let $D(V)$ denote the unit
disk in $V$.  A little disk is a (nonequivariant) affine map $D(V) \to
D(V)$.  We define $\D_{V}(U)_n$ as the space of $n$-tuples of
nonoverlapping little disks, where $G$ acts by conjugation on each
disk and $\Sigma_n$ in the obvious way.  The distinguished element
$1 \in \D_{V}(U)_1$ is the identity map and the structure maps are
induced from composition.  For $V \subseteq W$, there is a map induced
by taking the disk $v \mapsto av + b$ to the disk $w \mapsto aw + b$.
We define $\D(U) = \colim_V \D_V(U)$.

\item The {\emph{embeddings}} operad can be defined as follows.  Fix a
real representation $V \subset U$ with $G$-invariant inner product,
and let $\E(V)_{n}$ be the $G$-space of $n$-tuples of topological
embeddings $V \to V$ with disjoint image (topologized as a
$G$-subspace of the space of all embeddings with $G$ acting by
conjugation).  The distinguished element $1 \in \E(V)_1$ is the
identity map and the structure maps are induced by composition and
disjoint union.  As above, we can pass to the colimit over $V$.

\item The {\emph{Steiner}} operad $\K(U)$ is a (superior) variant of
the little disks operad $\D(U)$.  Fix a real representation $V \subset
U$ with $G$-invariant inner product.  Define $R_{V} \subset E(V)_1$ to
be the $G$-subspace of distance reducing embeddings $f \colon V \to
V$.  A Steiner path is a map $h \colon I \to R_V$ with $h(1) = \id$.
Let $P_V$ denote the $G$-space of Steiner paths (with $G$-action
coming from the action on $R_V$).  There is a natural projection map
$\pi \colon P_V \to R_V$ given by evaluation at $0$.  Define $\K(V)_n$
to be the $G$-space of $n$-tuples of Steiner paths $\{h_i\}$ such that
the projections $\pi(h_i)$ have disjoint images.  The Steiner operad
is defined to be $\K(U) = \colim_V \K(V)$.
\end{enumerate}
\end{definition}

\begin{remark}
The equivariant little disks operad is unfortunately extremely poorly
behaved; products of disks are not necessarily disks, and as observed
in~\cite[\S 3]{Mayrant}, the colimit over inclusions $V \subseteq W$
that defines $\D(U)$ is not compatible with the colimit of
$\Omega^V \Sigma^V$.  These problems are fixed by the Steiner operad,
and for these reasons the equivariant Steiner operad is preferable in
most circumstances.  Moreover, the Steiner operad is necessary for
capturing multiplicative structures (i.e., $E_\infty$ ring spaces) via
operad pairings --- there are equivariant operad pairs
$\big(\K(V), \cL(V)\big)$ for each $V \subset U$ and $(\K_U, \cL_U)$.
In contrast, it does not seem possible to have an operad pairing
involving the little disks operad.  See~\cite[10.2]{guilloumay} for
further discussion of this point.
\end{remark}

We have the following result about the $G$-homotopy type of the little
disks and Steiner operads~\cite[9.7, 10.1]{guilloumay}.

\begin{proposition}
Let $V \subset U$ be a real representation with $G$-invariant inner
product.  Then the $n$th spaces $\D(V)_n$ and $\K(V)_n$ are
$G \times \Sigma_n$-equivalent to the equivariant configuration space
$F(V, n)$.
\end{proposition}

Passing to colimits, this has the following corollary:

\begin{corollary}
The $G$-operads $\D(U)$ and $\K(U)$ are $\Ninfty$ operads for any
universe $U$.
\end{corollary}

The classical argument about contractibility of the spaces of equivariant isometries shows the following lemma.

\begin{lemma}
The $G$-operad $\cL(U)$ is an $\Ninfty$ operad for any universe $U$.
\end{lemma}


One of our original motivations for this paper was to understand the
relationship between $\D(U)$ or $\K(U)$ and $ \cL(U)$ in the case of a
general universe $U$.  We give an answer in the spirit of Lewis'
beautiful work relating dualizability of an orbit $G/H$ to whether it
embeds in the universe $U$~\cite{lewisincomplete}.  The surprising
conclusion of our study will be just how far apart $\K(U)$ and
$ \cL(U)$ can be for an incomplete universe $U$; see
Section~\ref{sec:LUDU}.

\subsection{Indexing systems}\label{ssec:Indexing}

There is a close connection between our $\Ninfty$ operads and certain
subcategories of the categories of finite $G$-sets.  However, as is
often the case in equivariant homotopy, we never want to consider just
the group $G$; instead we should consider all subgroups on equal
footing. This motivates the following replacement for a category.

\begin{definition}
A {\emph{categorical coefficient system}} is a contravariant functor
$\underline{\cC}$ from the orbit category of $G$ to the category of
small categories.
\end{definition}

As we will almost never be talking about abelian group valued
coefficient systems in this paper, we will often abusively drop the
prefix ``categorical''.

\begin{definition}
A {\emph{symmetric monoidal coefficient system}} is a contravariant
functor $\underline{\cC}$ from the orbit category of $G$ to the
category of symmetric monoidal categories and strong symmetric
monoidal functors.

If $\underline{\cC}$ is a symmetric monoidal coefficient system, then
the {\emph{value at H}} is $\underline{\cC}(G/H)$, and will often be
denoted $\underline{\cC}(H)$.

For a symmetric monoidal coefficient system $\underline{\cC}$, let
\[
i_{H}^{\ast}\colon\underline{\cC}(G)\to\underline{\cC}(H)
\]
denote the restriction map associated to the natural map $G/H\to G/G$.
\end{definition}

We can also consider ``enriched'' coefficient systems that take values
in enriched categories.  Most of the naturally arising categories in
equivariant homotopy actually sit in enriched symmetric monoidal
coefficient systems.

\begin{definition}
Let $\underline{\Top}_{(-)}$ be the enriched coefficient system of
spaces. The value at $H$ is $\Top_{H}$, the category of $H$-spaces and
all (not just equivariant) maps. Similarly, let
$\underline{\Top}^{(-)}$ be the associated ``level-wise fixed
points'', the value at $H$ is $\Top^{H}$, the category of $H$-spaces
and $H$-maps. There are two compatible symmetric monoidal structures:
disjoint union and Cartesian product.

Let $\underline{\Sp}_{(-)}$ be the enriched coefficient system of
spectra. The value at $H$ is $\Sp_{H}$, the category of $H$-spectra
and all maps. Let $\underline{\Sp}^{(-)}$ be the associated
coefficient system whose value at $H$ is the category of $H$-spectra
and $H$-maps. We again have two symmetric monoidal structures we can
consider: wedge sum and smash product.
\end{definition}

The most important category for our study of $\Ninfty$ operads is the
coefficient system of finite $G$-sets.

\begin{definition}
Let $\mSet$ be the symmetric monoidal coefficient system of finite
sets. The value at $H$ is $\Set^{H}$, the category of finite $H$-sets
and $H$-maps. The symmetric monoidal operation is disjoint union.
\end{definition}

We will associate to every $\Ninfty$ operad a subcoefficient system of
$\mSet$.  The operadic structure gives rise to additional structure on
the coefficient system.

\begin{definition}
We say that a full sub symmetric monoidal coefficient system $\cF$ of
$\mSet$ is {\emph{closed under self-induction}} if whenever
$H/K\in \cF(H)$ and $T\in \cF(K)$, $H\times_{K} T\in \cF(H)$.
\end{definition}  

\begin{definition}
Let $\cC\subset\cD$ be a full subcategory. We say that $\cC$ is a
{\emph{truncation subcategory}} of $\cD$ if whenever $X\to Y$ is monic
in $\cD$ and $Y$ is in $\cC$, then $X$ is also in $\cC$.

A truncation sub coefficient system of a symmetric monoidal
coefficient system $\underline{\cD}$ is a sub coefficient system that
is levelwise a truncation subcategory.
\end{definition}
In particular, for finite $G$-sets, truncation subcategories are those
which are closed under passage to subobjects.

\begin{definition}
An {\emph{indexing system}} is a truncation sub symmetric monoidal
coefficient system $\mF$ of $\mSet$ that contains all trivial sets and is closed under self
induction and Cartesian product.
\end{definition}

\begin{definition}\label{def:poset}
Let $\Coef(\Set)$ be the poset of all subcoefficient systems of
$\mSet$, ordered by inclusion.  Let $\poset$ be the poset of all
indexing systems.
\end{definition}

With this, we can state our main result describing the homotopy
category of $\Ninfty$ operads.

\begin{theorem}
There is a functor
\[
\mC\colon \Nop \to \poset
\]
which descends to a fully-faithful embedding of categories
\[
\mC\colon \Ho(\Nop)\to\poset.
\]
\end{theorem}

\section{Admissible sets and $\Ninfty$ operads}\label{sec:Admissibles}
The construction of the functor $\mC$ proceeds in two steps. We first
define a functor, also called $\mC$, from symmetric sequences with an
analogous universal property for their constituent spaces to the poset
$\Coef(\Set)$. We then show that if a symmetric sequence arises from
an operad, then the resulting value of $\mC$ actually lands in
$\poset$.

\subsection{Symmetric sequences and the functor $\mC$}
We begin looking very generally at what sorts of families of subgroups
can arise, using only at the universal space property of the spaces
in an $\Ninfty$ operad and the freeness of the $\Sigma_{n}$-action.

\begin{definition}
An $\Ninfty$ symmetric sequence is a symmetric sequence $\cO$ in
$G$-spaces such that for each $n$,
\begin{enumerate}
\item $\cO_{n}$ is a universal space for a family $\cF_{n}(\cO)$ of subgroups of $G\times\Sigma_{n}$ and 
\item $\Sigma_{n}$ acts freely.
\end{enumerate}
\end{definition}
In particular, the underlying symmetric sequence for an $\Ninfty$
operad is always of this form.

Our entire analysis hinges on a standard observation about the
structure of subgroups of $G\times\Sigma_{n}$ which intersect
$\Sigma_{n}$ trivially.

\begin{proposition}\label{prop:graph}
If $\Gamma\subset G\times\Sigma_{n}$ is such that $\Gamma\cap
(\{1\}\times\Sigma_{n})=\{1\}$, then there is a subgroup $H$ of $G$
and a homomorphism $f\colon H\to \Sigma_{n}$ such that $\Gamma$ is the
graph of $f$.
\end{proposition}


Thus the subgroup $\Gamma$ is equivalent to an
$H$-set structure on $\underline{n}=\{1,\dots,n\}$.  It will be
essential to our future analysis to recast the whole story in terms of
$H$-sets.

\begin{definition}
For an $H$-set $T$, let $\Gamma_{T}$ denote the graph of the
homomorphism $H\to\Sigma_{|T|}$ defining the $H$-set structure.  We
write that an $H$-set $T$ is {\emph{admissible}} for $\cO$ if
$\Gamma_{T}\in\cF_{|T|}(\cO)$.
\end{definition}

The requirements associated to the stipulation that $\cF_{\ast}(\cO)$
forms a family (closure under subgroups and conjugacy) translates to
the following observation in terms of admissibility:

\begin{proposition}\label{prop:admitfam}
If an $H$-set $T$ of cardinality $n$ is admissible, then
\begin{enumerate}
\item for all subgroups $K\subset H$, $i_{K}^{\ast}(T)$ is admissible, and
\item the $gHg^{-1}$-set $g\cdot T$ is admissible.
\item every $H$-set isomorphic to $T$ (as an $H$-set) is admissible.
\end{enumerate}
\end{proposition}

Proposition~\ref{prop:admitfam} actually shows that the admissible
sets assemble into a sub coefficient system of $\mSet $. This allows
us to define the functor $\mC$.

\begin{definition}
Let $\mC(\cO)$ denote the full subcoefficient system of $\mSet$ whose
value at $H$ is the full subcategory of $\Set_{H}$ spanned by the
admissible $H$-sets.
\end{definition}

\begin{proposition}
If $\cO\to\cO'$ is a map of $\Ninfty$ symmetric sequences, then
\[
\mC(\cO)\subset\mC(\cO').
\]
\end{proposition}
\begin{proof}
Let $T$ be an admissible set for $\cO$. By definition, this means that
$\cO_{|T|}^{\Gamma_{T}}\neq\emptyset$. Since we have a
$G\times\Sigma_{|T|}$-equivariant map $\cO_{|T|}$ to $\cO'_{|T|}$, we
know that the $\Gamma_{T}$ fixed points of $\cO'_{|T|}$ cannot be
empty.
\end{proof}

To refine our map, we recall the relevant notion of weak equivalence
for $G$-symmetric sequences.

\begin{definition}
A map $f\colon \cO\to\cO'$ between $G$-symmetric sequences is a weak
equivalence if for each $n$ it induces a weak equivalence of
$G \times \Sigma_n$ spaces.
\end{definition}

Notice that a weak equivalence of $\Ninfty$ operads give rise to an
underlying equivalence of $\Ninfty$-symmetric sequences.  Unpacking
the definition immediately gives the following proposition.

\begin{proposition}
If $f\colon \cO\to\cO'$ is a weak equivalence between
$\Ninfty$-symmetric sequences, then $\mC(\cO)=\mC(\cO')$.
\end{proposition}

\subsection{Symmetric monoidal structure of $\mC(\cO)$ and the operadic structure}\label{sec:Properties}

For an $\Ninfty$ operad $\cO$, the spaces $\cO_{n}$ do not exist in
isolation, and the structure maps on $\cO$ assemble to show that
$\mC(\cO)$ has extra structure. We first show that $\mC(\cO)$ is never
empty.

\begin{proposition}\label{prop:TrivialSets}
For all subgroups $H$ and for all finite sets $T$ of cardinality $n$,
the trivial $H$-set $T$ is admissible.
\end{proposition}

\begin{proof}
This follows from condition (iii) of Definition~\ref{def:geinfop}.
\end{proof}

\begin{lemma}\label{lem:DisjointUnion}
The coefficient system $\mC(\cO)$ is closed under (levelwise)
coproducts, and is thus a symmetric monoidal subcoefficient system of
$\mSet$.
\end{lemma}

\begin{proof}
We give the proof for the case of $S\amalg T$; other cases are
analogous. Let $m_{1}=|S|$ and $m_{2}=|T|$.  By definition, the fact
that $S$ and $T$ are admissible $H$-sets means that there exist
subgroups $\Gamma_1 \subset G \times \Sigma_{m_1}$ and
$\Gamma_2 \subset G \times \Sigma_{m_2}$ which are the graphs of
homomorphisms
\[
f_1 \colon H \to \Sigma_{m_1} \qquad \textrm{and} \qquad f_2 \colon
H \to \Sigma_{m_2}
\]
respectively.

Since $\aO$ is an operad, we know there exists a composition map
\[
\gamma \colon \aO_2 \times \aO_{m_1} \times \aO_{m_2} \to \aO_{m_1 + m_2}
\]
which is at least $G \times
(\{e\} \times \Sigma_{m_1} \times \Sigma_{m_2})$ equivariant.  Let
$\Gamma \subset G \times \Sigma_{m_1+m_2}$ be the subgroup specified
by the graph
\[
\Gamma = \{(h, f_1(h) \amalg f_2(h)) \, | \, h \in H\}.
\]
Consider the map $\gamma^{\Gamma}$ induced by passage to fixed points.
On the left hand side, by hypothesis we know that the fixed points are
contractible --- this is true for $\aO_{m_1}$ and $\aO_{m_2}$ by
admissibility, and for $\aO_{2}$ by
Proposition~\ref{prop:TrivialSets}.  Therefore, $\aO_{m_1 +
m_2}^{\Gamma}$ cannot be empty and is therefore contractible.
Translating, this means precisely that $S \amalg T$ is an admissible
$H$-set.
\end{proof}

Already we have neglected structure on the category of finite
$G$-sets. In addition to the disjoint union, there is a Cartesian
product. This is a form of the disjoint union, however, as $G/K\times
G/H$ is the ``disjoint union'' of $G/H$ indexed by the $G$-set $G/K$:
\[
G/K\times G/H \cong \coprod_{G/K}G/H,
\]
where $G$ acts on both the indexing set and the summands.
Induction has a similar formulation as an indexed coproduct, and our
admissible sets are closed under some forms of each operation.

\begin{lemma}\label{lem:CartesianProduct}
For each $H$, the category $\cC_{H}(\cO)$ is closed under Cartesian
product, and thus $\mC(\cO)$ inherits the structure of a symmetric
bimonoidal category levelwise.
\end{lemma}

\begin{proof}
Without loss of generality, we may assume that $H=G$, and let $S$ be
an admissible $G$-set of cardinality $m$ and $T$ one of cardinality
$n$. Associated to $S$ is a subgroup $\Gamma_{S}$ which is the graph
of $f\colon G\to\Sigma_{m}$, and associated to $T$, we have a similar
subgroup $\Gamma_{T}$ and function $h\colon G\to\Sigma_{n}$. Now there
is an embedding
\[
\Delta\colon\Sigma_{m}\times\Sigma_{n}\to\Sigma_{m}\wr\Sigma_{n}
\] 
which is just the diagonal on the $\Sigma_{n}$ factor, and we let
$F\colon G\to\Sigma_{m}\wr\Sigma_{n}$ be $\Delta\circ(f\times
h)$. Finally, let $\Gamma_{S\times T}$ be the graph of $F$.

We need now to show two things:
\begin{enumerate}
\item that $(\cO_{m}\times\cO_{n}^{m})^{\Gamma_{S\times T}}$ is non-empty (which in turn forces the $\Gamma_{S\times T}$-fixed points of $\cO_{mn}$ to be non-empty) and
\item that the function $F$ classifies the $G$-set $S\times T$.
\end{enumerate}

For the first part, we observe that $\Gamma_{S\times T}$ acts on
$\cO_{m}\times\cO_{n}^{m}$ via its natural action on the two named
factors. Thus
\[
(\cO_{m}\times\cO_{n}^{m})^{\Gamma_{S\times
T}}=\cO_{m}^{\Gamma_{S\times T}}\times (\cO_{n}^{m})^{\Gamma_{S\times
T}}.
\]

The action on the $\cO_{m}$ term factors through the canonical
quotient map
\[
G\times\Sigma_{m}\wr\Sigma_{n}\to G\times\Sigma_{m},
\] 
and the image of $\Gamma_{S\times T}$ under this quotient map is
$\Gamma_{S}$. By assumption, $\cO_{m}^{\Gamma_{S}}$ is contractible,
and hence so is $\cO_{m}^{\Gamma_{S\times T}}$.

The action on the second factor is slightly more complicated. We make
the following observation: the diagonal map $\cO_{n}\to\cO_{n}^{m}$ is
$(G\times\Sigma_{m}\times\Sigma_{n})$-equivariant, where $\Sigma_{m}$
acts trivially on the first factor and where we have identified
$\Sigma_{m}\times\Sigma_{n}$ with its image under $\Delta$. The group
$\Gamma_{S\times T}$ is contained in the subgroup
$G\times \Image(\Delta)$, and so the diagonal map is $\Gamma_{S\times
T}$-equivariant. By constructions, the action of $\Gamma_{S\times T}$
on $\cO_{n}$ is via $\Gamma_{T}$, and we therefore have fixed
points. This implies that $\cO_{n}^{m}$ has $\Gamma_{S\times T}$-fixed
points as well.

For the second part, we make a simple observation: in the arrow
category of finite sets, the automorphism group of the canonical
projection $S\times T\to S$ is isomorphic to
$\Sigma_{m}\wr\Sigma_{n}$. The $\Sigma_{m}$ acts by permuting the
base, and then the $\Sigma_{n}^{m}$ acts as the automorphisms of the
fibers. By our construction of $F$, the resulting $G$-set is the one
in which the base is the $G$-set $S$, and where all of the fibers are
the $G$-set $T$.
\end{proof}

\begin{lemma}\label{lem:SelfInduction}
The symmetric monoidal coefficient system $\mC(\cO)$ is closed under
self-induction.
\end{lemma}

\begin{proof}
Without loss of generality, we may assume $H=G$, as for the proof
given, we may simply replace all instances of $G$ with $H$. Now assume
that $G/K$ is in $\cC_{G}(\cO)$, and let $T$ be in $\cC_{K}(\cO)$. Let
$n$ be the cardinality of $T$, and let $m$ be the index of $K$ in $G$.

Associated to $T$ is a homomorphism $\pi\colon K\to \Sigma_{n}$, and
by assumption, $\cO_{n}^{\Gamma_{T}}\simeq \ast$. Finally, let
$g_{1},\dots, g_{m}\in G$ be a complete set of coset representatives
for $G/K$, and let $\sigma\colon G\to\Sigma_{m}$ be the homomorphism
induced by the left action of $G$ on $G/K$. Again, by assumption,
$\cO_{m}^{\Gamma_{G/K}}\simeq \ast$.

To prove the result, we must explicitly describe the induced set
$G\times_{K}T$. The argument is standard. Since $\{g_{1},\dots,
g_{m}\}$ is a complete set of coset representatives of $G$, for $1\leq
i\leq n$, we have a homomorphism
\[
\big(\sigma,(k_1,\dots,k_n)\big)\colon G\to \Sigma_{m}\wr K,
\]
where $\sigma$ and each of the functions $k_{i}$ are defined by
\[
g\cdot g_{i}=g_{\sigma(i)} k_{i}(g).
\]
The homomorphism $G\to\Sigma_{nm}$ describing the induced set
$G\times_{K}T$ arises from this homomorphism via the map $\pi$:
\[
Ind(g)=\Big(\sigma(g),\big(\pi(k_{1}(g)), \dots,\pi(k_{m}(g))\big)\Big)\in \Sigma_{m} \wr \Sigma_{n}.
\]

We need to now analyze the fixed points of $\Gamma$, the graph of
$Ind$, on $\cO_{m}\times\big(\cO_{n}\big)^{m}$. The group
$G\times\Sigma_{m}\wr\Sigma_{n}$ acts independently on $\cO_{m}$ and
on $\cO_{n}^{m}$. On $\cO_{m}$, it acts via the canonical quotient to
$G\times\Sigma_{m}$, and on $\cO_{n}^{m}$, $G$ acts diagonally while
$\Sigma_{m}\wr\Sigma_{n}$ has the obvious action. Thus
\[
\Big(\cO_{m}\times\big(\cO_{n}\big)^{m}\Big)^{\Gamma}=\cO_{m}^{\Gamma_{G/K}}\times\big(\cO_{n}^{m}\big)^{\Gamma}.
\]
It will suffice to show that these fixed points are non-empty. The
first factor is actually contractible, by assumption, so we need only
produce a fixed point for the second factor. Since the
$\Gamma_{T}$-fixed points of $\cO_{n}$ are non-empty, we can find a
point $x\in\cO_{n}$ such that
\[
(k,\pi k)\cdot x=x
\]
for all $k\in K$. Then we quickly show that
\[
y=\big((g_{1},1)\cdot x,\dots,(g_{m},1)\cdot x\big)
\] 
is a $\Gamma$-fixed point. To streamline the typesetting, let
$\sigma=\sigma(g)$, and $k_{i}=k_{i}(g)$, and let
\[
\gamma=\Big(g,\sigma,\big(\pi(k_{1}), \dots,\pi(k_{m})\big)\Big).
\]
Then we have a chain of equalities
\begin{align*}
\gamma\cdot y
&= (g,1)\cdot \Big(\big(g_{\sigma^{-1}(1)},\pi
k_{\sigma^{-1}(1)} \big)\cdot x,\dots, \big(g_{\sigma^{-1}(m)},\pi
k_{\sigma^{-1}(m)}\big)\cdot x\Big) \\ &=
\Big(\big(g\cdot g_{\sigma^{-1}(1)},\pi k_{\sigma^{-1}(1)}\big)\cdot x,\dots, \big(g\cdot g_{\sigma^{-1}(m)},\pi k_{\sigma^{-1}(m)}\big)\cdot x\Big) \\
&=
\Big((g_{1},1)\big(k_{\sigma^{-1}(1)},\pi k_{\sigma^{-1}(1)}\big)\cdot x,\dots,(g_{m},1)\big(k_{\sigma^{-1}(m)},\pi k_{\sigma^{-1}(m)}\big)\cdot x\Big) \\
&= \big((g_{1},1)\cdot x,\dots,(g_{m},1)\cdot x\big)=y.
\end{align*}

Thus we conclude that $\big(\cO_{m}\times(\cO_{n})^{m}\big)^{\Gamma}$
is non-empty, and therefore so is $\cO_{nm}^{\Gamma}$.
\end{proof}

One way to package
Lemmata~\ref{lem:DisjointUnion}, \ref{lem:CartesianProduct},
and \ref{lem:SelfInduction} is via the $G$-symmetric monoidal
structure on the category of finite $G$-sets. Induction is actually a
special kind of disjoint union: we simply allow the group $G$ to act
on the indexing set (in this case $G/H$) for the disjoint
union. Working more generally, we see that we can easily make sense of
a disjoint union of $(-)$-sets $S_{t}$ indexed by a $G$-set $T$
provided
\begin{enumerate}
\item $S_{t}$ is a $Stab(t)$-set and
\item $S_{g\cdot t}$ is in bijective correspondence with $S_{t}$ and the action of $g$ intertwines the $Stab(t)$ and $gStab(t)g^{-1}$ actions.
\end{enumerate}
Our lemmas can then be repackaged in this language.

\begin{corollary}\label{cor:GDisjointUnion}
If $T\in\cC_{G}(\cO)$ and if for all $t\in T$, we have an admissible
$Stab(t)$-set $S_{t}$ satisfying the compatibility condition above,
then
\[
\coprod_{t\in T}S_{t}\in \cC_{G}(\cO).
\]
\end{corollary}

\begin{warning}
While it is true that $\mC(\cO)$ forms a coefficient system and is
closed under some indexed coproducts, it is not true that $\mC(\cO)$
is always closed under {\emph{arbitrary}} induction (making it a kind
of category-valued Mackey functor). The norm machinery described in
Section~\ref{sec:NormsandInduction} can be used to produce operads
which close up $\mC(\cO)$ under certain inductions.
\end{warning}

Thus far we have used only the composition structure of the operad
(and hence, all of this would work in a non-unital context). For the
last piece of structure, we must have a unital algebra.

\begin{lemma}\label{lem:Summands}
The coefficient system $\mC(\cO)$ is a truncation subcoefficient
system of $\mSet$: if $Z=S\amalg T$ is an admissible $G$-set, then
both $S$ and $T$ are admissible.
\end{lemma}

\begin{proof}
We use the unit map to show this. The admissibility of $Z$ shows that
there is a map $f\colon G\to \Sigma_{|Z|}$ and
$\cO_{|Z|}^{\Gamma_{Z}}\simeq \ast$.  The disjoint union decomposition
of $Z$ into $S\amalg T$ shows that we can choose this map to factor
through the inclusion
$\Sigma_{|S|}\times\Sigma_{|T|}\subset\Sigma_{|Z|}$ (in fact, the
subgroup $\Gamma_{Z}$ corresponding to $Z$ probably does not have this
property; however, a conjugate of $\Gamma_{Z}$ will). In this case,
the projection of $\Gamma_{Z}$ onto $G\times\Sigma_{|S|}$ realizes the
subgroup $\Gamma_{S}$ corresponding to $S$, and similarly for $T$.

We now use the composition and the identity to deduce the
result. Consider the composition:
\[
\cO_{|Z|}\times\cO_{1}^{|S|}\times\cO_{0}^{|T|}\to\cO_{|S|}.
\]
This map is $(G\times\Sigma_{|S|}\times\Sigma_{|T|})$-equivariant,
where on the first factor, the action is via the obvious inclusion and
where the action on the target is via the quotient to
$G\times\Sigma_{|S|}$. Since the map defining the $G$-action on $Z$
factors through $\Sigma_{|S|}\times\Sigma_{|T|}$, the group
$\Gamma_{Z}$ is actually a subgroup of
$G\times\Sigma_{|S|}\times\Sigma_{|T|}$. The $\Gamma_{Z}$-action on
$\cO_{|S|}$ is via the quotient $\Gamma_{|S|}$, so
\[
\cO_{|S|}^{\Gamma_{Z}}=\cO_{|S|}^{\Gamma_{S}}.
\]
Since the spaces in the operad are universal spaces for a family, it
will again suffice to show that
\[
\big(\cO_{|Z|}\times\cO_{1}^{|S|}\times\cO_{0}^{|T|}\big)^{\Gamma_{Z}}=\cO_{|Z|}^{\Gamma_{Z}}\times(\cO_{1}^{|S|}\times\cO_{0}^{|T|})^{\Gamma_{Z}}\neq\emptyset.
\]

By assumption, the first factor is non-empty. For the second, the
diagonal map
\[
\cO_{1}\times\cO_{0}\to\cO_{1}^{|S|}\times\cO_{0}^{|T|}
\]
is $\Sigma_{|S|}\times\Sigma_{|T|}$-equivariant, with the image being
the fixed points. The space $\cO_{1}\times\cO_{0}$ is
$G$-equivariantly contractible, so we know that in fact
\[
\emptyset\neq(\cO_{1}^{|S|}\times\cO_{0}^{|T|})^{G\times\Sigma_{|S|}\times\Sigma_{|T|}}\subset (\cO_{1}^{|S|}\times\cO_{0}^{|T|})^{\Gamma_{Z}}.\qedhere
\]
\end{proof}

\begin{corollary}
The coefficient system $\mC(\cO)$ is closed under finite limits.
\end{corollary}

\begin{proof}
Equalizers are subobjects in $\mSet$, and
Lemma~\ref{lem:CartesianProduct} shows that each category is also
closed under finite products.
\end{proof}

Putting together all of these lemmas, we deduce the following theorem.

\begin{theorem}\label{thm:FunctorC}
The functor $\mC$ is a functor from the homotopy category of $\Ninfty$
operads to the poset $\poset$.
\end{theorem}


\subsection{Application: Linear isometries and little disks}\label{sec:LUDU}

We pause here to provide a surprising application: for all but three
finite groups $G$, there are universes $U$ such that the linear
isometries and little disks (or Steiner) operads associated to $U$ are
inequivalent. To show this, we need only apply our functor $\mC$.

\begin{theorem}\label{thm:LUFamily}
For the equivariant linear isometries operad on $U$, the admissible
$H$-sets are those $T$ such that there is an $H$-equivariant embedding
\[
\bZ[T]\otimes U\to U.
\]
\end{theorem}

\begin{proof}
In fact, the statement of the theorem is a restatement of definition
of the linear isometries operad. If $T$ is an admissible $H$-set, then
by definition
\[
\cL(U^{\oplus n},U)^{\Gamma_{T}}=\cL_{\Gamma_{T}}(U^{\oplus n},U)\neq \emptyset.
\]
The group $\Gamma_{T}$ acts on $U$ via the quotient $H$. The only
question is how it acts on
\[
U^{\oplus n}=\bZ\{1,\dots,n\}\otimes U.
\]
On the tensor factor $U$, the $\Gamma_{T}$-action is again via the
quotient $H$. On the other tensor factor, by the definition of $T$,
the $\Gamma_{T}$-action is the $H$-action on $\bZ[T]$. This gives the
result.
\end{proof}

The truncation and disjoint union conditions on our indexing sets
shows that admissibility is completely determined by the admissibility
of orbits $H/K$. The condition for admissibility for $\cL(U)$ then is
that there is an $H$-equivariant embedding
\[
\Ind_{K}^{H}i_{K}^{\ast}U\to i_{H}^{\ast}U.
\]
This requirement is actually a ``cofamily'' condition in $H$: if $K$
is subconjugate to some $K'$ in $H$, then $\bZ[H/K']\otimes U$
$H$-embeds into $U$ whenever $\bZ[H/K]\otimes U$ does.

\begin{theorem}\label{thm:DUFamily}
For the equivariant little disks operad on $U$, the admissible
$H$-sets are those $T$ such that there is an $H$-equivariant embedding
\[
T\to U.
\]
\end{theorem}

\begin{proof}
This is essentially due to Lewis. An embedding of $T$ into $U$ can be
fattened into a tiny equivariant neighborhood of $T$ embedded into
$U$. This is an embedding of $T\times D$ into $U$ which is
$H$-equivariant, and this is exactly what an element of the
$\Gamma_{T}$-fixed points of
\[
\D\left(\coprod_{1}^{n}D,D\right)=\D(\{1,\dots,n\}\times D,D)
\]
looks like. Just as in the linear isometries case, the existence of a
single embedding is sufficient to have a contractible space.
\end{proof}

\begin{corollary}
For any universe $U$, there is a map in the homotopy category of operads
\[
\cL(U)\to \D(U).
\]
\end{corollary}
\begin{proof}
For any finite $H$-set $T$, $T$ always $H$-equivariantly embeds into
$\mathbb R\{T\}$.  Thus if $T$ is admissible for $\cL(U)$, then it is 
also admissible for $\D(U)$.
\end{proof}

Since the condition on the category $\mC(\cL(U))$ described in
Theorem~\ref{thm:LUFamily} is much more stringent than the one for the
category $\mC(\D(U))$ described in Theorem~\ref{thm:DUFamily},  there is, a priori, no reason that the operads need be the same for a particular universe. We will show that in fact, they can be
different (and for most groups, hugely different, as explained in
Theorem~\ref{thm:VeryDifferent} below). We first show an important
example in which they coincide.

\begin{theorem}
If $N$ is a normal subgroup in $G$ and if $U_{N}$ is the universe
generated by $\bR[G/N]$, then $ \cL(U_{N})$ and $\D(U_{N})$ are
equivalent.
\end{theorem}
This universe is the $N$-fixed points of the complete universe, and
this statement should be viewed as an analogue of the symmetric
monoidal embedding of $G/N$-spectra in $G$-spectra.

\begin{proof}
We just have to show that the admissible sets are the same in both
cases, and these are the sets with stabilizer containing $N$. Since
$N$ is normal in $G$, there is no difference between restricting to
$H$ and restricting to $HN$, and in this case, $U_{N}$ restricts to
$U_{N}$ but with $G$ replaced by $HN$. It therefore suffices to look
at those $G$ sets which are admissible.

The admissible $G$-sets for $\D(U_{N})$ are those with stabilizer $H$
such that $G/H$ embeds in $U_{N}$.  Since $G$ is finite,
\[
\Emb_{G}(G/H,U)=U^{H}-\bigcup_{H<K} U^{K},
\]
where $H<K$ means $H$ is properly subconjugate to $K$. For all
subgroups $H$, the $H$ fixed points are equal to the $HN$-fixed
points, and so if $H$ does not contain $N$, there are no embeddings of
$G/H$ into $U$. On the other hand, if $H$ does contain $N$, then the
transfer shows that the $H$-fixed points of $U_{N}$ is the universe
generated by $\bR[G/H]$. This visibly contains $G/H$. Thus the
admissible $G$-sets for $\D(U_{N})$ are those with stabilizers
containing $N$.

For $ \cL(U)$, we need only to determine those $H$ such that
$\bZ[G/H]\otimes U_{N}$ embeds in $U_{N}$. The universe $U_{N}$ has
the defining feature that all of $N$ fixes $U_{N}$ and such that
larger subgroups of $G$ move points. If $H$ contains $N$, then the
desired condition obviously holds. If $H$ does not contain $N$, then
$N$ does not fix $\bZ[G/H]$, and therefore, there are no embeddings.

Thus in both cases, the admissible $G$-sets are precisely those whose
stabilizers contain $N$, and $\D_{n}(U_{N})$ and $\cL_{n}(U_{N})$ are
equivalent.
\end{proof}

We can now prove the main result in this subsection: for all but two
groups, there are universes $U$ such that $\cL(U)$ and $\D(U)$ are
inequivalent.

\begin{theorem}\label{thm:Failing}
If $G$ is a finite group of order bigger than $3$, then there is a
universe $U$ such that $\cL(U)$ and $\cD(U)$ are not equivalent.
\end{theorem}

The proof follows immediately from a small, representation theoretic
lemma.

\begin{lemma}\label{lem:FaithfulSubRep}
If $G$ is a finite group of order bigger than $3$, then there is a
representation $V$ such that
\begin{enumerate}
\item $G$ embeds into $V$, and
\item there is a non-trivial irreducible representation $W$ of $G$ such that $W$ is not a summand of $V$.
\end{enumerate}
\end{lemma}
In fact, $V$ can be chosen as a faithful representation.
\begin{proof}
First note that if $G$ is a simple group of order at least $5$, then
every non-trivial representation of $G$ is faithful. By the class
equation, there are more than $2$ non-trivial irreducible complex
representations, and hence at least $2$ non-trivial, irreducible real
representations. Any one such representation will satisfy the
conditions of the lemma.


Now assume that $N$ is a non-trivial, proper normal subgroup of
$G$. Let $\bar{\rho}_{N}$ denote the quotient of the real regular
representation of $N$ by the trivial summand. Then we claim that
$V=\Ind_{N}^{G}\bar{\rho}_{N}$ satisfies the conditions of the
lemma. 

The reduced regular representation is faithful and induction
preserves this property. Since the representation is faithful, the collection of all vectors with non-trivial stabilizer is a union of proper hyperplanes of $V$, and since $G$ is finite, this is a proper subset of $V$. Thus $G$ embeds into $V$.

For the second condition, let $\lambda$
denote a non-trivial real representation of the quotient group
$G/N$. Frobenius reciprocity shows that there are no non-trivial maps
between complexifications of $V$ and $\lambda$ (since the restriction
of $\lambda$ to $N$ is always trivial), and thus $\lambda$ is not a
summand of $V$.
\end{proof}

\begin{proof}[Proof of Theorem~\ref{thm:Failing}]
Let $V$ be a faithful representation of $G$ satisfying the conditions
of Lemma~\ref{lem:FaithfulSubRep}, and let $U=\infty(1+V)$. Then by
assumption, $G$ embeds into $U$, so $G/\{e\}$ is an admissible $G$-set
for $\cD(U)$. However, $U$ is not the infinite regular representation,
since $V$ does not contain every irreducible representation of $G$,
and so $G/\{e\}$ is not an admissible $G$-set for $\cL(U)$.
\end{proof}

If $G$ has order $2$ or $3$, then this will fail: there are only two
irreducible real representations: the trivial one and multiplication
by the corresponding root of unity. Thus in these cases there are only
two universes: the trivial universe and the complete universe.

With slightly more care, we can refine the above theorem.

\begin{theorem}\label{thm:VeryDifferent}
If $G$ is not simple, then there is a universe $U$ such that $\D(U)$
is not equivalent to $\cL(W)$ for any universe $W$.
\end{theorem}

\begin{proof}
Let $N$ be a non-trivial, proper normal subgroup of $G$, and let
$V=\Ind_N^G\bar{\rho}_N$ as in the proof of
Lemma~\ref{lem:FaithfulSubRep}. Since this is a faithful
representation of $G$, we know that $G$ embeds in $V$. If $G$ is
admissible for $\cL(W)$, then Theorem~\ref{thm:LUFamily}, $W$ must be
the compete universe. In particular, $\mC(\cL(W))=\mSet$.

We prove the theorem by showing that $G/N$ itself does not embed in
$U=\infty(1+V)$. Obviously, any such embedding lands in a finite
subrepresentation, so we show that $G/N$ does not embed into $k(1+V)$
for any $k$. A map of $G$-sets
\[
G/N\to k(1+V),
\]
is the same as an $N$-fixed point of $k(1+V)$. However, since $N$ is a
normal subgroup,
\[
i_N^\ast V=[G:N]\bar{\rho}_N,
\]
which has no fixed points. Thus any map lands entirely in the trivial
factor, and hence is a constant map on $G/N$.
\end{proof}

\begin{remark}
We do not know for which simple groups Theorem~\ref{thm:VeryDifferent}
holds. For cyclic groups of prime order, it fails: there are only two
indexing systems, the trivial and complete one, both of which
correspond to little disks and linear isometries operads. For $A_{2n+1}$,
the restriction of the quotient of the defining representation for
$\Sigma_{2n+1}$ by the trivial summand generates a universe in which
$A_{2n+1}/D_{4n+2}$ does not embed, showing that for $A_{2n+1}$,
Theorem~\ref{thm:VeryDifferent} holds.
\end{remark}

\section{The homotopy category of $\Ninfty$ operads}\label{sec:hocat}

In this section, we show that the functor $\mC$ is a fully-faithful
embedding and explain why we believe that it is fact an equivalence.

\subsection{Faithfulness}

We begin by recording some easy results about the relationships
between coefficient systems that correspond to natural constructions
on operads.

\begin{proposition}
If $\cO$ and $\cO'$ are $\Ninfty$ operads, then $\cO\times\cO'$ is an
$\Ninfty$ operad, and
\[
\mC(\cO\times\cO')=\mC(\cO)\cap\mC(\cO').
\]
\end{proposition}

\begin{proof}
The only part that requires any proof is the second part; the operadic
properties are straightforward. The second part is actually a standard
observation in equivariant homotopy theory: if $E\cF$ and $E\cF'$ are
universal spaces for families $\cF$ and $\cF'$ respectively, then
$E\cF\times E\cF'$ is a universal space for $\cF\cap\cF'$. This
follows immediately from consideration of the fixed points. The
translation to the categorical version is then as above.
\end{proof}

\begin{corollary}\label{cor:OperadComparison}
If $\mC(\cO)\subset\mC(\cO')$, then the natural projection
\[
\cO\times\cO'\to\cO
\]
is a weak equivalence.
\end{corollary}
\begin{proof}
For all $n$, both $(\cO\times\cO')_{n}$ and $\cO_{n}$ are universal
spaces for the same family of subgroups.
\end{proof}

\begin{corollary}
If $\mC(\cO)=\mC(\cO')$, then in the homotopy category, $\cO$ and
$\cO'$ are isomorphic.
\end{corollary}
\begin{proof}
Apply Corollary~\ref{cor:OperadComparison} twice to the zig-zag $\cO
\from \cO \times \cO' \to \cO'$.
\end{proof}

\begin{corollary}
If $\mC(\cO)\subset\mC(\cO')$, then in the homotopy category, we have
a map
\[
\cO\to\cO'.
\]
\end{corollary}

In order to go further, we calculate the derived space of maps between
two operads $\cO$ and $\cO'$.

\begin{proposition}\label{prop:mapcontract}
The derived mapping space from any $G$-operad $\aO$ to an $N_\infty$
operad $\aO'$ is either empty or contractible.
\end{proposition}

\begin{proof}
We perform the calculation in the category of $G$-operads in
simplicial sets.  Since $G$ is discrete, there is a model structure on
$G \times \Sigma_n$-simplicial sets where the weak equivalences and
fibrations are detected on passage to fixed point spaces (and the
cofibrations are the monomorphisms)~\cite[3.1.9]{Rezk}.  Let
$\Coll_{\gsset}$ denote the category of symmetric sequences of
$G$-simplicial sets.  Since this is equivalent to the product (over
$n \geq 0$) of the categories of $G \times \Sigma_n$-simplicial sets,
there is a levelwise model structure on $\Coll_{\gsset}$ in which the
weak equivalences and fibrations are detected pointwise.  The
forgetful functor from the category of $G$-operads in simplicial sets
to $\Coll_{\gsset}$ has a left adjoint free functor, and the transfer
argument of~\cite[3.2.10]{Rezk} applies to lift the model structure on
$\Coll_{\gsset}$ to one on $G$-operads in simplicial sets.  Note that
these model structures are simplicial and cofibrantly-generated.

Let $\Gop(\aT)$ denote the category of $G$-operads in topological
spaces and let $\Gop(\sset)$ denote the category of $G$-operads in
simplicial sets.  The geometric realization and singular complex
functors preserve products and so induce an adjoint pair
\[
\Sing \colon \Gop(\aT) \rightleftarrows \Gop(\sset) \colon |-|.
\]  
Furthermore, since both of these functors preserve weak
equivalences~\cite[3.1.10]{Rezk}, we can compute the derived mapping
space in either category.  More precisely, the fact that $|-|$ and
$\Sing$ preserve equivalences and are such that the unit and co-unit
of the adjunction are natural weak equivalences implies that there is
a weak equivalence
\[
L^H \Gop(\aT) (\aO, \aO') \htp L^H \Gop(\sset)
(\Sing \aO, \Sing \aO'),
\]
where $L^H$ denotes the Dwyer-Kan simplicial mapping space.

This latter can be computed as the internal mapping space in the model
category of operads in $G$-simplicial sets after replacing the source
with a cofibrant object and the target with a fibrant object.  In this
model structure a cofibrant replacement of a $G$-operad can be
computed as a retract of a cell operad.  Moreover, the fibrant objects
are precisely the levelwise fibrant objects and so in particular
$\Sing \aO'$ is fibrant.

Thus, we can compute the mapping space by resolving the $G$-operad
$\Sing \aO$ as a cell object.  That is, $\Sing \aO = \colim_n X_n$,
where each stage $X_n$ can be described as the (homotopy) pushout
\[
\xymatrix{
\mathbb{F} A_n \ar[r] \ar[d] & X_{n-1} \ar[d] \\
\mathbb{F} B_n \ar[r] & X_{n}.
}
\]  
Here $\mathbb{F}$ is the free functor from $\Coll_{\gsset}$ to
$\Gop(\sset)$.  Therefore, there is an equivalence
\[
\Map(\Sing \aO,-) \htp \holim_n \Map(X_n, -).
\]
It now suffices to show that $\Map(X_n, -)$ is contractible.  Inductively, we can use the
pushout description of $X_n$ above to reduce to the case of free
$G$-operads.  Finally, observe that maps from a free operad into any
$N_\infty$ operad are contractible or empty: by adjunction, they are
computed on the level of symmetric sequences, and any $N_\infty$
operad is made up of universal spaces.
\end{proof}


\begin{corollary}
The functor $\mC$ is a faithful embedding of $\Ho(\Nop)$ into
$\poset$.
\end{corollary}

\subsection{Towards fullness}\label{sec:fullness}

We now explain why we believe that in fact $\mC$ is an equivalence of
categories.  We will use the categorical Barratt-Eccles operad of
Guillou-May~\cite[2.3]{guilloumay}.  To produce operads in spaces, we
simply take the geometric realization of the nerve.

\begin{definition}
The categorical Barratt-Eccles operad is defined by
\[
\mathbb O_{n}=i^*Map(G,\Sigma_{n}),
\]
where $i^*\colon \Set\to\Cat$ is the right-adjoint to the ``object''
functor.

The operadic structure maps are simply induced by the embeddings of
products of symmetric groups into bigger ones.
\end{definition}

The functor $i^*$ assigns to each set the category whose objects are
the set and for which there is a unique morphism in each direction
between any pair of objects.

\begin{remark}
The operad $\mathbb O$ is the norm from trivial categories to
$G$-categories of the Barratt-Eccles operad $\underline{\Sigma}$,
defined by
\[
\underline{\Sigma}_{n}=i^*\Sigma_{n}.
\]
From this perspective, it is immediate that $\mathbb O_{n}$ has fixed
points for all subgroups $H$ of $G\times\Sigma_{n}$ for which
$H\cap\Sigma_{n}=\{e\}$.
\end{remark}

Associated to an element of $\Coef(\Set)$ is a collection of families
$\cF_{n}$ of subgroups of $G\times\Sigma_{n}$: $T$ is an $H$-set in
our coefficient system if and only if $\Gamma_{T}$ is in
$\cF_{|T|}$. Using this, we can build a sub-symmetric sequence in
categories of $\mathbb O$.

\begin{definition}
If $\cF_{\ast}$ is a sequence of families of subgroups of
$G\times\Sigma_{\ast}$, then let
\[
\mathbb O^{\cF}_{n}=i^*\{f\in\mathbb O_{n} \,|\, Stab(f)\in\cF_{n}\}.
\]
\end{definition}

Since the family is closed under conjugation, for each $n$, $\mathbb
O^{\cF}_{n}$ is a $G\times\Sigma_{n}$-subcategory of $\mathbb O$. By
construction, the geometric realization of $\mathbb O^{\cF}$ is an
$\Ninfty$ symmetric sequence, and similarly, we immediately have the
following.

\begin{proposition}
Let $\cF_{\ast}$ be the sequence of families of subgroups associated
to an $\Ninfty$ symmetric sequence $\cO$. Then we have
\[
\mC(|\mathbb O^{\cF}|)=\mC(\cO).
\]
\end{proposition}

We make the following conjecture, which would establish an equivalence
of categories between $\Ho(\Nop)$ and $\poset$.

\begin{conjecture}\label{thm:Surjective}
If $\cC$ is an indexing system and if $\cF$ is the associated sequence
of families of subgroups, then $\mathbb O^{\cF}$ is a sub-operad of
$\mathbb O$.
\end{conjecture}

An interesting question (about which we do not have a conjectural
answer) is whether or not all homotopy types in $\Nop$ are realized by
the operads that ``arise in nature'', i.e., the equivariant Steiner
and linear isometries operads.

\section{The structure of $\Ninfty$-algebras}\label{sec:Algebras}

Although we can consider algebras over an $\Ninfty$ operad $\aO$ in
any symmetric monoidal category enriched over $G$-spaces, we are most
interested in the examples of orthogonal $G$-spectra with the smash
product and $G$-spaces with the Cartesian product.  In both of these
examples, the notion of weak equivalence of operads given in
Definition~\ref{defn:weak} is validated by the fact that a weak
equivalence of $\Ninfty$ operads induces a Quillen equivalence of the
associated categories of algebras.  (See
Appendix~\ref{sec:monadicalgebras} for details.)  Therefore, the
associated data of the coefficient system captures all of the relevant
structure.  We now turn to describing this structure in geometric
terms.

Specifically, the name $\Ninfty$ refers to the additional structure
encoded by an $\Ninfty$ operad: norms, or more precisely indexed
products.  In spectra with the smash product, these arise as the
Hill-Hopkins-Ravenel norm, and the operadic structure encodes the
analogue of the counit of the adjunction between the norms and the
forgetful functors for commutative ring spectra.  In spaces with the
Cartesian product, these arise as coinduction, and the operadic
structure maps encode the transfer in algebras over the Steiner
operads.

In the following definition, we use the technical device of exploiting
the equivalence of categories between orthogonal $G$-spectra on the
complete universe and orthogonal $G$-spectra on a trivial
universe~\cite[\S VI.1]{MM}, as pioneered in the Hill-Hopkins-Ravenel
construction of the norm.  Specifically, given an orthogonal
$G$-spectrum $X$ on a complete universe, we forget to the trivial
universe, perform the construction indicated in the formula, and then
left Kan extend back to the complete universe.

\begin{definition}
Let $T$ be an $G$-set.
\begin{enumerate}
\item If $E$ is an orthogonal $G$-spectrum, then let
\[
N^{T}E=\left(G\times\Sigma_{|T|}/\Gamma_{T+}\right)\wedge_{\Sigma_{|T|}}
E^{\wedge |T|}.
\]
\item If $X$ is a $G$-space, then let
\[
N^{T}X=\left(G\times\Sigma_{|T|}/\Gamma_{T}\right)\times_{\Sigma_{|T|}}
X^{\times |T|}.
\]
\end{enumerate}
\end{definition}

As stated, there is a potential conflict of notation --- $N^{T}E$
could refer to the preceding definition or to the Hill-Hopkins-Ravenel
norm. This ambiguity is resolved by the following proposition, which
uses the fact that $G$-spaces and orthogonal $G$-spectra are tensored
over $G$-spaces.  If $X$ and $Y$ are $G$-spaces, we write $F(X,Y)$ to
denote the space of all continuous maps from $X$ to $Y$, given the
conjugation $G$-action.

\begin{proposition}
Let $T$ be an $H$-set.
\begin{enumerate}
\item  Let $E$ be an orthogonal $G$-spectrum. Then a decomposition
$T=\coprod_{i} H/K_{i}$ gives a homeomorphism
\[
\left(G\times\Sigma_{|T|}/\Gamma_{T+}\right)\wedge_{\Sigma_{|T|}} E^{|T|}\cong G_{+}\wedge_{H}\bigwedge_{i} N_{K_{i}}^{H}i_{K_{i}}^{\ast} E,
\]
where $N_{K_{i}}^{H}$ is the Hill-Hopkins-Ravenel norm.
\item Let $X$ be a $G$-space. Then we have a homeomorphism
\[
\left(G\times\Sigma_{|T|}/\Gamma_{T}\right)\times_{\Sigma_{|T|}} X^{\times |T|}\cong G\times_{H}F(T,X).
\]
\end{enumerate}
\end{proposition}

\begin{proof}
The first statement is essentially the definition of the norm. The second follows
immediately from the Cartesian product endowing $G$-spaces with a
symmetric monoidal structure.
\end{proof}

\begin{proposition}
The assignments
\[
(T,E)\mapsto N^{T}(E)
\quad\text{ and }\quad
(T,X)\mapsto N^{T}(X)
\]
specify strong symmetric monoidal functors in both factors, and
moreover we have natural homeomorphisms
\[
N^{S\times T}(E)\cong N^{S}N^{T}(E)
\quad\text{ and }\quad
N^{S\times T}(X)\cong N^{S}N^{T}(X).
\]
\end{proposition}

\begin{proof}
The first part is immediate from the definition. For the second,
unpacking Lemma~\ref{lem:CartesianProduct} makes the above
isomorphisms very clear. The identification of the subgroup of
$\Sigma_{|S\times T|}$ associated to $\Gamma_{S\times T}$ shows that
the two sides are the same.
\end{proof}

\subsection{The structure of $\cO$-algebras}

We focus on the general structure of $\cO$-algebras in $G$-spaces and
orthogonal $G$-spectra.  For brevity of exposition, we will describe
all of our maps and structure for orthogonal $G$-spectra herein, using
the smash product. Everything we say holds {\emph{mutatis mutandis}}
for $G$-spaces using the Cartesian product.

We start with the most basic structure: an algebra over an $\Ninfty$
operad looks like an ordinary, classical algebra over a
non-equivariant $E_{\infty}$ operad.

\begin{proposition}\label{prop:Naive}
If $R$ is an $\cO$-algebra in spectra, then $R$ is a naive
$E_{\infty}$ ring spectrum in the sense that $R$ has a multiplication
that is unital, associative, and commutative up to all higher
homotopy.
\end{proposition}

\begin{proof}
Choose an ordinary, non-equivariant $E_{\infty}$ operad $\cE$ and
endow it with a trivial $G$-action. Since $\mC(E)$ is the initial
object in $\poset$, we know that we have a map from (an operad
equivalent to) $\cE$ to $\cO$.  Thus any $\cO$-algebra is by
restriction an $\cE$-algebra.
\end{proof}

The other admissible sets appear as extra structure. 

\begin{construction}\label{cons:OperadMaps}
For an orthogonal $G$-spectrum $E$ and $T$ an admissible $H$-set for
$\cO$ with associated subgroup $\Gamma_{T}$, then by definition of
admissibility, we are given a $(G\times\Sigma_{|T|})$-contractible
space of maps
\[
(G\times\Sigma_{|T|})/\Gamma_{T}\to\cO_{|T|},
\]
and smashing over $\Sigma_{|T|}$ with $E^{\wedge |T|}$ yields a
contractible space of maps
\[
G_{+}\wedge_{H}N^{T}i_{H}^{\ast}E\to \cO_{|T|+}\wedge_{\Sigma_{|T|}}E^{\wedge
|T|}.
\]
\end{construction}

This contractible space of maps gives us extra structure for an
$\cO$-algebra.

\begin{lemma}\label{lem:Norms}
If $R$ is an $\cO$-algebra and $T$ is an admissible $H$-set, then
there is a contractible space of maps
\[
G_{+}\wedge_{H}N^{T}i_{H}^{\ast}R\to R
\]
built from the maps of Construction~\ref{cons:OperadMaps}.
\end{lemma}

\begin{proof}
The maps in question are the composite
\[
G_{+}\wedge_{H}N^{T}i_{H}^{\ast}R\to \cO_{|T|+}\wedge_{\Sigma_{|T|}}R^{\wedge
|T|}\to R,
\]
where the first map is any of the maps in
Construction~\ref{cons:OperadMaps} arising from the contractible space
\[
F_{G\times\Sigma_{|T|}}(G\times\Sigma_{|T|}/\Gamma_{T},\cO_{|T|})=\cO_{|T|}^{\Gamma_{T}}.\qedhere
\]
\end{proof}

\begin{remark}
By convention, we assume that the empty set is always admissible. In
this case, we can again construct a contractible space of maps
\[
G_{+}\wedge_{H} N^{\emptyset}i_{H}^{\ast}R\to R,
\]
since by assumption, $N^{\emptyset} i_{H}^{\ast} R$ is the symmetric
monoidal unit.
\end{remark}

We can strengthen these results.  Recall that the category of algebras
over an $E_{\infty}$ operad is homotopically tensored over finite sets
in the sense that given an algebra $R$ and a map $T\to S$ of finite
sets, we have a contractible space of maps $R^{|T|}\to R^{|S|}$
encoding the multiplication.  An analogous result holds in this
context, where here the algebras over an $\Ninfty$ operad $\cO$ are
homotopically tensored over $\cC_{G}(\cO)$.

\begin{theorem}\label{thm:NormsasTransfers}
If $T$ and $S$ are admissible $G$-sets and $f\colon T\to S$ is a
$G$-map, then for any $\cO$-algebra $R$, we can construct a
contractible space of maps $N^{T}R\to N^{S}R$ encoding the
multiplication.
\end{theorem}

\begin{proof}
For $S$ a trivial $G$-set, this is the content of
Lemma~\ref{lem:Norms}. For the general case, we observe that a general
map between $G$ sets can be written as a disjoint union of surjective
maps onto orbits inside $S$. Disjoint unions correspond to external
smash products, and hence, it suffices to consider $S$ a single orbit
and $T\to S$ surjective. This, however, can be rewritten as
\[
T\to \coprod_{|T/G|}S\to S,
\]
where the first map is the disjoint union of the surjection restricted
to each orbit of $T$ and the second is just the fold map. It will
therefore suffice to show two things:

\begin{enumerate}
\item That associated to the fold map we can construct a contractible
space of maps, and
\item associated to a surjective map $G/H\to G/K$, we can construct a
contractible space of maps.
\end{enumerate}

The fold map in turn is just $S$ times the fold map sending $|T/G|$
points with trivial $G$ action to a single point. We have a
contractible space of maps
\[
R^{\wedge |T/G|}\to R
\]
by Lemma~\ref{lem:Norms} again, applied to the trivial $G$-set. Taking
the norm $N^{S}(-)$ of these produces the required contractible space
of maps for the fold.

Now consider $T=G/H$ and $S=G/K$. By possibly composing with an
automorphism of $T$, we may assume that $H$ is a subgroup of $K$ and
that the map is the canonical quotient. In this case, the map is
\[
G\times_{K}(K/H\to K/K).
\]
Since $K/H$ is a summand of $i_{K}^{\ast}(G/H)$, we know that $K/H$ is
an admissible $K$-set. Lemma~\ref{lem:Norms} gives us a contractible
space of maps
\[
N^{K/H}(i_{K}^{*}R)\to i_{K}^{*}R.
\]
Applying the functor $N_{K}^{G}$ produces a contractible space of maps
\[
N^{G/H}(R)\to N^{G/K}(R),
\]
as required.
\end{proof}

\begin{remark}\label{rem:OperadsonGSets}
One way of interpreting Theorem~\ref{thm:NormsasTransfers} is that
equivariant operads should really be indexed on finite $G$-sets, not
just (a skeleton of) finite sets.  Such a definition is very natural
using the perspective on $\infty$-operads developed in Lurie~\cite{HA}
--- instead of working with fibrations over Segal's category $\Gamma$,
equivariant $\infty$-operads should be defined as fibrations over the
equivariant analogue $\Gamma_G$.  We intend to return to explore this
perspective in future work.
\end{remark}

\begin{corollary}\label{cor:Functorial}
If $S$, $S'$, and $S''$ are finite admissible $G$-sets, and
\[
S\xrightarrow{f} S'\xrightarrow{f'} S''
\]
are maps of $G$-sets and if $R$ is an $\cO$-algebra, then for any
choice of maps coming from Theorem~\ref{thm:NormsasTransfers}, the
following diagram commutes up to homotopy
\[
\xymatrix{{N^{S}} R\ar[r]^{f_{\sharp}} \ar[rd]_{(f'\circ f)_{\sharp}} & {N^{S'}R} \ar[d]^{f'_{\sharp}} \\
& {N^{S''}R.}}
\]
\end{corollary}

\begin{theorem}\label{thm:ExistenceofNorms}
An $\cO$-algebra $R$ is an orthogonal $G$-spectrum with maps
\[
G_{+}\wedge_{H}N^{T}i_{H}^{\ast}R\to R
\]
for all admissible $H$-sets $T$ such that the following conditions
hold.
\begin{enumerate}
\item For all admissible $G$-sets $S$ and $T$, the following diagram
homotopy commutes
\[
\xymatrix{
{N^{S\amalg T}Y\simeq N^{S}R\wedge N^{T}R}\ar[r]\ar[rd] & {R\wedge
R}\ar[d] \\ {} & {R.}}
\]
\item For all admissible $G$-sets $S$ and $T$, the following diagram
homotopy commutes

\[
\xymatrix{
{N^{S\times T}R\simeq N^{S}N^{T}R}\ar[r]\ar[rd] & {N^{S}R}\ar[d] \\ {}
& {R.}}
\]
\item For all admissible sets $S$ and $T$ such that for some $K\subset
G$, $i_{K}^{\ast}(S)\cong i_{K}^{\ast}(T)$, the following diagram
homotopy commutes
\[
\xymatrix{
{i_{K}^{\ast}N^{S}R\simeq N^{i_{K}^{\ast}S}i_{K}^{\ast}
R}\ar@{<->}[rr] \ar[dr] & & {N^{i_{K}^{\ast}T}i_{K}^{\ast}R\simeq
i_{K}^{\ast}N^{T}R}\ar[dl] \\ {} & {i_{K}^{\ast}R.} & {}}
\]
\end{enumerate}
In fact, all of these diagrams commute up to coherent homotopy; this
coherence data is precisely the information encoded by the operad.
\end{theorem}

The first two conditions express compatibility with the multiplication
and with the other norms. The third part shows that the structure is
well-behaved upon passage to fixed points. We spell out a short,
illuminating, example.

\begin{example}

Let $G=C_{2}$ (although any finite group will work here), and let
$\cO$ denote an $\Ninfty$ operad weakly equivalent to the Steiner
operad on the complete universe.  By assumption, $\cO_{2}$ is the
universal space $E_{C_{2}}\Sigma_{2}$ for $\Sigma_{2}$-bundles in
$C_{2}$-spaces. If we let $\rho_{2}$ denote the regular representation
of $C_{2}$ and $\tau$ denote the sign representation of $\Sigma_{2}$,
then a cofibrant model for $\cO_{2}$ is given by
\[
S\big(\infty (\rho_{2}\otimes\tau)\big)=\lim_{\rightarrow}
S(n\rho_{2}\otimes\tau).
\]

Inside of this is of course $S(\rho_{2}\otimes\tau)$. This has a cell
structure given by
\[
\big((C_{2}\times\Sigma_{2})/C_{2}\amalg (C_{2}\times\Sigma_{2})/\Delta\big)\cup_{f} (C_{2}\times\Sigma_{2})\times e^{1},
\]
where $\Delta$ is the diagonal copy of $C_{2}=\Sigma_{2}$, and $f$ is
the canonical quotient
\[
f\colon (C_{2}\times\Sigma_{2})\times
S^{0}=(C_{2}\times\Sigma_{2})\amalg (C_{2}\times\Sigma_{2})\to
(C_{2}\times\Sigma_{2})/C_{2}\amalg (C_{2}\times\Sigma_{2})/\Delta.
\]

Thus if we have an $\cO$-algebra $R$, then the zero cells together
give a map
\[
R^{\wedge 2}\vee N^{C_{2}}i_{e}^{\ast} R\to R,
\]
while the attaching map for the one-cell identifies the restriction of
the map on the first factor with the restriction of the map on the
second factor.
\end{example}

\subsection{Norms, coinductions, and cotensors of $\Ninfty$ operads}\label{sec:NormsandInduction}

We now describe the behavior of $\Ninfty$ operads and
characterizations of their collections of admissible sets under
various standard functors.  Our basic tool is the following standard
result:

\begin{proposition}\label{prop:lax}
Let $F \colon \aC \to \aD$ be a lax symmetric monoidal functor between
symmetric monoidal categories $\aC$ and $\aD$.  Given an operad $\aO$
in $\aC$, then $F \aO$ is an operad in $\aD$, and $F$ induces a
functor
\[
\aC[\aO] \to \aD[F\aO]
\]
connecting the categories of $\aO$-algebras and $F\aO$-algebras.
\end{proposition}

\begin{proof}
The fact that $F \aO$ forms an operad in $\aD$ is a standard
consequence of regarding operads as monoids in symmetric sequences;
e.g., see~\cite[3.3]{westerland} for a more detailed discussion.  To
see that $F$ induces a functor on algebras, it suffices to exhibit a
natural map
\[
(F \aO) FX \to F(\aO X)
\]
in $\aD$, where $(F \aO) X$ denotes the free $F \aO$-algebra on
$X$. Writing this out, we want a natural map
\[
\coprod^{\infty}_{n = 0} F\aO(n) \otimes_{\Sigma_n} (FX)^{n} \to
F\left(\coprod^{\infty}_{n=0} \aO(n) \otimes_{\Sigma_n} X^n\right).
\]
The lax symmetric monoidal structure of $F$ induces a composite
\[
F\aO(n) \otimes (FX)^{n} \to F\aO(n) \otimes F(X^n) \to
F(\aO(n) \otimes X^n),
\]
and now we map this into the orbits and then the coproduct.  By the
universal property of the coproduct, as $n$ varies these maps assemble
into the desired map.
\end{proof}

\subsubsection{Coinduction and $\Ninfty$ operads}

Just as restriction of an $\Ninfty$ operad is again an $\Ninfty$
operad, coinduction preserves the collection of $\Ninfty$ operads. 

\begin{definition}
If $\cO$ is an $H$-$\Ninfty$ operad, then let $N_{H}^{G}\cO=F_{H}(G,\cO)$ be the
$\Ninfty$ operad defined by
\[
F_{H}(G,\cO)_{n}=F_{H}(G,\cO_{n})\cong
F_{H\times\Sigma_{n}}(G\times\Sigma_{n},\cO_{n}).
\]
\end{definition}

These spaces assemble into an operad using the diagonal map on $G$ to
see that coinduction is lax symmetric monoidal.  The last equality
shows that this is actually a universal space for a family of
subgroups of $G\times\Sigma_{n}$.  Identifying the family is fairly
straightforward and lets us identify the admissible sets.

\begin{proposition}
For any finite group $G$, if $\cF$ is a family of subgroups of
$H\subset G$, then $F_{H}(G,E\cF)$ is a universal space for the family
of subgroups of $G$ corresponding to the sieve
${i_{H}^{\ast}}^{-1} \Set_{\cF}$.
\end{proposition}
\begin{proof}
By the adjunction, for any finite $G$-set $T$ (in fact, for any
$G$-space), we have a homeomorphism
\[
F^{G}(T,F_{H}(G,E\cF))\cong F^{H}(i_{H}^{\ast} T,E\cF).
\]
This space is either contractible or empty according to whether
$i_{H}^{\ast} T$ is in $\Set_{\cF}$ or not, respectively.
\end{proof}

Specializing to the families which arise from an $\Ninfty$ operad, we
conclude the following.

\begin{proposition}
Let $\cO$ be an $H$-$\Ninfty$ operad. For any $K\subset G$, a $K$ set
$T$ is admissible if any only if for all $g\in G$,
\[
i_{H\cap gKg^{-1}} g\cdot T\in \cC(\cO)(H\cap gKg^{-1}).
\]
\end{proposition}
\begin{proof}
Let $n$ be the cardinality of a finite $K$ set $T$. Consider
$G\times\Sigma_{n}/\Gamma_{T}$. By the previous proposition, we need
only check that the restriction of this to $H\times\Sigma_{n}$ is in
the family associated to $\cO_{|T|}$. By the double coset formula,
this restriction is a disjoint union of $H\times\Sigma_{n}$-sets of
the form
\[
H\times\Sigma_{n}/\big(H\times\Sigma_{n}\cap (g,\sigma)\Gamma_{T}
(g,\sigma)^{-1}\big).
\]
The conjugates of $\Gamma_{T}$ are all again graphs of functions. In
this case, the conjugate of $\Gamma_{T}$ is the graph of the function
describing the $gKg^{-1}$-set $g\cdot T$ (with $\sigma$ here just
providing an isomorphism of this $gKg^{-1}$-set with
another). Intersecting this with $H\times\Sigma_{n}$ is again the
graph of a homomorphism, this one with domain $H\cap gKg^{-1}$. The
result follows.
\end{proof}

\begin{corollary}
An $H$-set $T$ is admissible for $F_{H}(G,\cO)$ if and only if it is
admissible for $\cO$.
\end{corollary}

\begin{corollary}
If $N$ is a normal subgroup of $G$, then the condition is simply that
a $K$-set is admissible if and only if its restriction to $N\cap K$ is
admissible.
\end{corollary}

\begin{corollary}
If $N$ is a normal subgroup of $G$, then $\cC(N_{N}^{G}\cO)$ is the
closure of $\cC(\cO)$ under the operation
\[
\Ind_{(-)\cap N}^{(-)}\colon \Set_{(-)\cap N}\to \Set_{(-)}.
\]
\end{corollary}

We can now explain the connection between norms of algebras and algebras over the norm of an $\Ninfty$ operad $\cO$. One of the defining features of the norm in spectra is a homeomorphism
\[
N_{H}^{G} \Sigma^{\infty}(X_{+}) \cong \Sigma^{\infty}\big(F_{H}(G,X)_{+}\big),
\]
which follows immediately from the fact that $\Sigma^{\infty}_{+}$ is
a symmetric monoidal functor from spaces with Cartesian product to
spectra with the smash product. Thus we expect a close connection
between algebras in spaces or spectra over an $\Ninfty$ operad and
those over its norm.  The following corollary is an immediate
consequence of Proposition~\ref{prop:lax}. 

\begin{corollary}\label{cor:NormsofAlgebras}
If $R$ is an $\cO$-algebra in spaces or spectra for an $\Ninfty$
$H$-operad $\cO$, then $N_{H}^{G}(R)$ is naturally a
$N_{H}^{G}\cO$-algebra.
\end{corollary}

\subsubsection{Cotensoring and $\Ninfty$ operads}

We close this subsection with a small result of independent interest:
cofree naive commutative $G$-ring spectra are automatically genuine
commutative $G$-ring spectra. This follows from the cotensoring
operation of spaces on $\Ninfty$ operads.

\begin{proposition}\label{prop:UniversalCoTensor}
Let $E$ be a universal space for a finite group $G$. If $X$ is a $G$-space, then
$
F(X,E)
$
is again a universal space for $G$.
\end{proposition}
\begin{proof}
A universal space is determined by the property that for any $G$-space $Y$, the space of $G$-equivariant maps
\[
F(Y,E)^{G}
\]
is either empty or contractible. Using the adjunction
\[
F\big(Y,F(X,E)\big)^{G}\cong F(Y\times X,E)^{G},
\]
we see that $F(X,E)$ again has the desired property.
\end{proof}

\begin{proposition}
If $X$ is a non-empty $G$-space, then for any $\Ninfty$ operad $\cO$, there is
an $\Ninfty$ operad $F(X,\cO)$ defined by 
\[
F(X,\cO)_n=F(X,\cO_n),
\]
where $X$ is viewed as a $G\times\Sigma_{n}$-space with trivial $\Sigma_{n}$ action and and with coordinatewise structure maps.
\end{proposition}

\begin{proof}
Since the cotensor is lax monoidal (using the diagonal map on $X$), Proposition~\ref{prop:lax} implies
that $F(X,\cO)$ forms an operad.  Proposition~\ref{prop:UniversalCoTensor} then implies that all space are universal spaces for some family of subgroups of $G\times\Sigma_{n}$. We need only show that $\Sigma_{n}$ acts freely.

Let $H\subset\Sigma_{n}$ be non-trivial, and consider the $H$-fixed points of $F(X,\cO_{n})$. Since $\Sigma_{n}$ acted trivially on $X$ and since $X$ was non-empty, the restriction of $X$ to $H$ is built entirely out of cells with stabilizer $H$. Since $i_{\Sigma_{n}}^{\ast}\cO_{n}=E\Sigma_{n}$, freeness follows.
\end{proof}

Naturality of the function object immediately gives the following
proposition.

\begin{proposition}\label{prop:NaturalityofOperadNorms}
If $f\colon X\to Y$ is a map of non-empty $G$-spaces, then
$f^{\ast}$ is a map of $G$-operads
\[
F(Y,\cO)\to F(X,\cO),
\]
and hence any $F(X,\cO)$-algebra is naturally a $F(Y,\cO)$-algebra.
\end{proposition}

In particular, the map to the terminal space $\ast$ shows that any
$F(X,\cO)$-algebra is naturally a $\cO$-algebra. 

When the $\Ninfty$ $H$-operad is the restriction of a $\Ninfty$
$G$-operad, then we can combine
Proposition~\ref{prop:NaturalityofOperadNorms}
and Corollary~\ref{cor:NormsofAlgebras}.

\begin{corollary}\label{cor:PreservationofAlgebra}
If $R$ is an $\cO$-algebra in spaces or spectra for a $\Ninfty$
$G$-operad $\cO$, then $N_{H}^{G}i_{H}^{\ast}(R)$ is again naturally
an $\cO$-algebra.

More generally, if $T$ is a finite $G$-set, then $N^{T}(R)$ is
naturally an $\cO$-algebra.
\end{corollary}

There is an extremely important (and somewhat surprising) case when $X
= EG$ --- the cotensor $F(EG,\aO)$ is then a genuine $G$-$E_{\infty}$ operad
for any $\aO$.  To make sense of this claim, consider the mapping
space $F(EG,E\Sigma_n)$, regarded as a $G \times \Sigma_n$-space where
$G \times \Sigma_n$ acts on $EG$ via the projection to $G$ and on
$E\Sigma_n$ via the projection to $\Sigma_n$.  Regarded as a universal
space for a family of subgroups of $G \times \Sigma_n$, $E\Sigma_n$
can admit maps only from spaces with isotropy contained entirely in
$G \times \{1\}$ --- but this is precisely the case for $EG$.

\begin{proposition}\label{prop:whoa}
For any $\Ninfty$ operad $\cO$, the $\Ninfty$ operad $F(EG,\cO)$ is a
$G$ $E_\infty$ operad. 
\end{proposition}

\begin{proof}
It suffices to show this for the trivial $\Ninfty$ operad $\cO^{tr}$,
since $F(EG,\cO^{tr})$ maps to $F(EG,\cO)$ for any other $\Ninfty$
operad $\cO$.

Let $\Gamma$ be any subgroup of $G\times\Sigma_n$ that intersects
$\Sigma_n$ trivially.  To show that the $\Gamma$ fixed points of the
cotensor are nonempty, by adjunction we need only show that
\[
(G\times\Sigma_n/\Gamma)\times EG
\]
can be built out of cells of the form $G/H\times\Sigma_n$. The
cellular filtration of $EG$ shows that it in turn suffices to show
that $G\times\Sigma_n$-equivariantly, we have an isomorphism
\[
(G\times\Sigma_n/\Gamma)\times (G\times\Sigma_{n}/\Sigma_{n})\cong \coprod G\times\Sigma_n
\]
This follows immediately from the equivalences
\begin{align*}
G\times (G\times\Sigma_n/\Gamma)\cong (G\times\Sigma_n/\{e\}\times\Sigma_n)\times (G\times\Sigma_n/\Gamma)\\ 
\cong G\times\Sigma_n\times_{\{e\}\times\Sigma_n} i_{\{e\}\times\Sigma_n}^{\ast} (G\times\Sigma_n/\Gamma).
\end{align*}
Since $\{e\}\times\Sigma_n$ is normal and since by assumption
\[
\Gamma\cap \{e\}\times\Sigma_{n}=\{e\},
\]
we have an equivariant isomorphism
\[
i_{\{e\}\times\Sigma_n}^{\ast} (G\times\Sigma_n/\Gamma) \cong \coprod_{|G/H|} \{e\}\times\Sigma_n,
\]
where $H$ is the image of $\Gamma$ under the projection to $G$.
\end{proof}

This now gives the following theorem.

\begin{theorem}
If $R$ is an algebra in orthogonal $G$-spectra over any $\Ninfty$
operad $\cO$, then the localized orthogonal $G$-spectrum
\[
F(EG_{+},R)
\]
is automatically an algebra over the terminal $\Ninfty$
operad. Moreover, the map  
\[
R\to F(EG_{+},R)
\]
is a map of $\cO$-algebras, where the target is an $\cO$-algebra by
the diagonal map $\cO\to F(EG_{+},\cO)$.  Analogous results hold for
an algebra over $\cO$ in $G$-spaces.
\end{theorem}

\begin{proof}
We give the proof for spectra; the case of spaces is analogous.
First, observe that $F(EG_+,R)$ is an algebra over the operad in
spectra specified by the cotensor $F(EG_+,\Sigma^{\infty}_+ \aO)$,
since $F(EG_+,-)$ is lax monoidal (using the diagonal map on $EG$).
Next, there is a natural map of operads
\[
\Sigma^{\infty} F(EG,\aO)_+ \to F(\Sigma^\infty
EG_+, \Sigma^\infty \aO_+)
\]
induced by the continuity of the functor $\Sigma^\infty (-)_+$.  The
first assertion now follows from Proposition~\ref{prop:whoa}, and the
second is immediate.
\end{proof}

\subsection{Multiplicative action maps}\label{sec:inter}

Based on the example of algebras over the commutative operad, one
expects that the operations parametrized by $\Ninfty$ operads are
multiplicative in the sense that for any point $o \in \aO(n)$, the
induced map 
\[
\mu_{o} \colon X^{\sma n} \to X
\] 
is itself a map of $\aO$-algebras, where the domain is given the
diagonal action of $\aO$.  More generally, we would expect this also
to hold equivariantly, where now the maps described in Lemma~\ref{lem:Norms}
and Theorem~\ref{thm:NormsasTransfers} are maps of appropriate
algebras. 

Classically, this situation is described via the formalism of
interchange of operads~\cite[\S 1]{Dunn}, which we review below.  To
study the case of Theorem~\ref{thm:NormsasTransfers}, wherein we
consider the norm of a map of $\cO$-algebras, we need to also address
the connection between algebras over the norm of an operad and the
norm of algebras over an operad.

Recall that given an object $X$ which is simultaneously an
$\aO$-algebra and an $\aO'$-algebra, we say that the two actions
interchange if for each point $x \in \aO_n$, the map $X^{n} \to X$ is
a map of $\aO'$-algebras and vice-versa.  We can express this
relationship by requiring that the diagram 
\[
\xymatrix{
(X^{n})^{m} \ar[r]^{\cong} \ar[d]^{\alpha^m} & (X^{m})^{n} \ar[r]^{\beta^n} & X^n \ar[d]^{\alpha} \\
X^m \ar[rr]^{\beta} & & X\\
}
\]
commute for each $\alpha \in \aO(n)$ and $\beta \in \aO'(m)$, where
the homeomorphism is given by the permutation that takes lexicographic
order to other lexicographic order.  

Interchange of operads is described by the tensor product of operads; by construction, $X$ is an $\aO$-algebra and an $\aO'$-algebra such that the actions interchange if and only if $X$ is an $\aO \otimes \aO'$-algebra~\cite[\S 1]{Dunn}.  The universal property of the tensor product of operads can also be described in terms of the
theory of pairings of operads~\cite{Maypairing} (see~\cite[\S 6.1]{guilloumay} for a discussion in the equivariant setting); a pairing 
\[
(\aO, \aO') \to \aO''
\] 
is a collection of suitable coherent maps $\aO_n \times \aO'_m \to \aO''_{nm}$.  In this language, the tensor product is the universal recipient for pairings.

The $\Ninfty$-condition is a homotopical one, parameterizing (as we
saw above) the ways to coherently multiply elements where we allow the
group to act on both the elements and on the coordinates. We therefore
expect that the tensor product of $\Ninfty$ operads will always be
$\Ninfty$: 

\begin{conjecture}
If $\cO$ and $\cO'$ are $\Ninfty$ operads, then (subject to suitably
cofibrancy conditions) $\cO\otimes \cO'$ is an $\Ninfty$ operad and moreover
\[
\mC(\cO\otimes\cO')=\mC(\cO)\vee\mC(\cO'),
\]
where $\vee$ denotes the least upper bound in the poset $\poset$.
\end{conjecture}

In particular, the conjecture implies that for any algebra over an
$\Ninfty$ operad $\cO$, the operad action interchanges with itself.

An immediate corollary of the definition of interchange is that when
the operadic action interchanges with itself, the maps in
Lemma~\ref{lem:Norms} are maps of $\cO$-algebras: 

\begin{proposition}
Let $R$ be an algebra over an $\Ninfty$ operad $\cO$, and assume that
the $\cO$-action interchanges with itself. Then for any surjective
maps $S\to T$ of admissible $H$-sets, the structure maps in
Theorem~\ref{thm:NormsasTransfers} 
\[
N^{S} i_{H}^{\ast}R\to N^{T}i_{H}^{\ast} R
\]
are maps of $N^{T}i_{H}^{\ast}\cO$-algebras.
\end{proposition}

We intend to return to a general analysis of the theory of the tensor
product of $G$-operads elsewhere.  However, for the cases of most
interest in applications, namely the equivariant Steiner and linear
isometries operads, it is possible to verify the necessary interchange
relations directly.

In~\cite[\S 10]{guilloumay}, it is shown that there is a pairing of
operads  
\[
\big(\K(V), \K(W)\big) \to \K(V \oplus W),
\] 
relying on an interchange map
\[
\theta \colon \K_n(U) \times \K_m(U) \to \K_{nm}(U \oplus U)
\]
that takes $n$ Steiner paths $\{k_1, \ldots, k_n\}$ and $m$ Steiner
paths $\{k'_1, \ldots , k'_m\}$ to the collection of the $nm$
product paths 
\[
k_i \times k_j' \colon I \to R_U \times R_U \subset R_{U \oplus U}.
\]
ordered lexicographically.  Choosing an equivariant homeomorphism
$U \oplus U \to U$, we deduce the following consequence:

\begin{proposition}\label{prop:Steiner}
Let $X$ be an algebra over the equivariant Steiner operad on $U$.
Then the operad action satisfies interchange with itself.  
\end{proposition}

\begin{corollary}\label{cor:TransfersAreLoopMaps}
If $X$ is an algebra over $\cK(U)$, then for any admissible $H$-set
$T$, the structure maps 
\[
N^{T}i_{H}^{\ast}X\to i_{H}^{\ast}X
\]
are maps of $\cK(U)$-algebras.
\end{corollary}

Essentially the same construction works for the linear isometries
operad.  To be precise, given $f \in \cL_n(U)$ and $g \in \cL_m(U)$,
we can decompose these into their components --- $f \colon U^n \to U$
gives rise to $f_1, f_2, \ldots, f_n \colon U \to U$ and $g \colon
U^m \to U$ gives rise to $g_1, g_2, \ldots, g_m \colon U \to U$.  The
interchange map here takes $\{f_i\}, \{g_i\}$ to the map 
\[
(U \oplus U)^{mn} \to U\oplus U
\] 
by the lexicographic pairings $\{f_i \oplus g_j\}$.
Therefore, using again a chosen homeomorphism $U \oplus U \to U$, we
have the following result.

\begin{proposition}
Let $R$ be an algebra over the equivariant linear isometries operad on $U$.
Then the operad action satisfies interchange with itself.  
\end{proposition}

\begin{corollary}\label{cor:NormsAreRingMaps}
If $R$ is an algebra over $\cL(U)$, then for any admissible $H$-set $T$, the structure maps
\[
N^{T}i_{H}^{\ast}R\to i_{H}^{\ast}R
\]
are maps of $\cL(U)$-algebras.
\end{corollary}

\section{$\Ninfty$-spaces and $\Ninfty$-ring spectra: Transfers and
norms}\label{sec:TransfersAndNorms}

In this section, we interpret the structure on algebras over $\Ninfty$
operads in the two cases of most interest: $G$-spaces and orthogonal
$G$-spectra.  In the former, the admissible sets control which
transfer maps exist; this provides a conceptual interpretation of the
way in which $\Ninfty$ operads controls the structure of equivariant
infinite loop spaces.  In the latter, the admissible sets control
which norms exist; this provides a conceptual interpretation of the
way in which $\Ninfty$ operads controls the structure of equivariant
commutative ring spectra.

\subsection{$\Ninfty$ algebras in spaces and the
transfer}\label{sec:transfers}

We begin by applying the machinery developed above to produce the
transfer in algebras over an $\Ninfty$ operad in spaces. The most
important examples of $\Ninfty$ operads from the point of view of
spaces are the equivariant Steiner operads $\cK(U)$, which model
equivariant infinite loop spaces. The goal of this section is to
describe how the transfer naturally arises from the operadic structure
maps.

In this section, we state our results in terms of an operad $\cO$ such
that the action of $\cO$ on any $\cO$-algebra $X$ interchanges with
itself.  (Recall that Proposition~\ref{prop:Steiner} tells us this is
true for $\cK(U)$.)  The following is a restatement of
Theorem~\ref{thm:NormsasTransfers} in the context of $G$-spaces.

\begin{theorem}\label{thm:GeneralTransfers}
If $\cO$ is an $\Ninfty$ operad, $S$ and $T$ are admissible $H$-sets,
and $f\colon T\to S$ is an $H$-map, then for any $\cO$-algebra $X$ in
$G$-spaces, we have a contractible space of maps 
\[
F(T,i_{H}^{\ast}X)\to F(S,i_{H}^{\ast}X),
\]
and if the map $f$ is surjective, then any choice is homotopic to a map of $N^{S}(\cO)$-algebras.
\end{theorem}

Applying fixed points and passing to homotopy groups, we produce
interesting maps:

\begin{theorem}\label{thm:TransfersinSpaces}
If $S$ and $T$ are admissible $H$-sets, if $f\colon T\to S$ is an
$H$-map, and if $X$ is an $\cO$-algebra in $G$-spaces, then there is
unique, natural (in $X$) map of abelian groups 
\[
f_{\ast}\colon\pi_{k}\big(F(T,i_{H}^{\ast}X)^{H}\big)\to \pi_{k}\big(F(S,i_{H}^{\ast}X)^{H}\big)
\]
for all $k\geq 0$.
\end{theorem}
\begin{proof}
Without loss of generality, we may assume that $H=G$.  If the map $f$
is not surjective, then we may use the splitting 
\[
S=Im(f)\amalg S'
\]
to produce a decomposition
\[
F(S,X)\cong F(Im(f),X)\times F(S',X).
\]
Proposition~\ref{prop:Naive} guarantees that for all $k\geq 0$, the
induced decomposition on homotopy groups of fixed points is a
splitting of abelian monoids.  Our map $f_{\ast}$ is the composite of
the map induced by $f\colon T\to Im(f)$ with the inclusion of the
summand associated to $S'$. We therefore may assume that $f$ is
surjective. 

Since the spaces in our operad are contractible, there is a unique homotopy class for the structure map given by Theorem~\ref{thm:GeneralTransfers}
\[
f_{\sharp}^{G}\colon F(T,X)^{G}\to F(S,X)^{G},
\]
which gives rise to a unique map of homotopy groups:
\[
f_{\ast}\colon \pi_{k}\big(F(T,X)^{G}\big)\to \pi_{k}\big(F(S,X)^{G}\big).
\] 
Proposition~\ref{prop:Naive} guarantees that for all $k\geq 0$, the homotopy groups of all fixed points of $N^{T}(X)$ are abelian monoids. It is obvious that $f_{\ast}$ is a map of abelian groups for $k\geq 1$. Since we may assume that the $f_{\sharp}$ comes from a surjective map,  our interchange assumption guarantees that the map $f_{\sharp}$ is a map of $N^{S}(\cO)$-algebras. Thus $f_{\ast}$ is a map of abelian monoids for all $k$.
\end{proof}

\begin{corollary}
If $H/K$ is an admissible $H$-set, then associated to the canonical
projection map  
\[
\pi_{K}^{H}\colon H/K\to H/H
\]
we have a natural map of abelian monoids
\[
tr_{K}^{H}=\pi_{K*}^{H}\colon \pi_{k} X^{K}\to \pi_{k}X^{H}.
\]
\end{corollary}

This map has the feel of the transfer map: on homotopy groups, we have
a map that goes from the fixed points for a subgroup back to the fixed
points for a larger group.  We shall shortly verify that upon passage
to spectra that this does give the usual transfer.  Before doing so,
we deduce some very nice structural corollaries from
Theorem~\ref{thm:ExistenceofNorms}.

\begin{proposition}\label{prop:DoubleCoset}
If $H/K$ is an admissible $H$-set, then the double coset formula determining the restriction of $tr_{K}^{H}$ to any subgroup $K'$ of $H$ holds:
\[
res_{K'}^{H}tr_{K}^{H}=\bigoplus_{g\in K'\backslash H/K} tr_{K'\cap gKg^{-1}}^{K'} res_{K'\cap gKg^{-1}}^{K}.
\]
\end{proposition}

This proposition, often called the ``Mackey double coset formula'' really has a simpler interpretation: the restriction to a subgroup $K'$ of the transfer associated to an $H$-set $T$ is the transfer associated to the $K'$-set $i_{K'}^{\ast} T$. As such, this is an immediate consequence of Theorem~\ref{thm:ExistenceofNorms} (iii).

\begin{corollary}
For an $\cO$-algebra $X$ for which the $\cO$-action interchanges with itself, the abelian group valued coefficient system
\[
\underline{\pi_{k}(X)}\colon\mSet\to\Ab
\]
defined by
\[
(T\in\Set^{H})\mapsto \pi_{k}\big(F(T,X)^{H}\big)
\]
has transfers for any admissible sets. 
\end{corollary}
These are therefore incomplete Mackey functors, studied by Lewis during his analysis of incomplete universes \cite{LewisHurewicz, LewisChange}. 

\begin{remark}
The forgetful functor on abelian group valued coefficient systems has a right adjoint: coinduction. By the universal property of the product, we have a natural isomorphism
\[
\underline{\pi_{k} F(G/H,X)}\cong \CoInd_{H}^{G}i_{H}^{\ast} \underline{\pi_{k}(X)}.
\]
This can be further simplified, using the construction of coinduction:
\[
\CoInd_{H}^{G}i_{H}^{\ast}\mM(T)=\mM(G/H\times T).
\]
This final formulation has an obvious extension to more general $G$-sets than orbits, and we follow Lewis's notation
\[
\mM_{S}(T)=\mM(S\times T)
\]
for a fixed $G$-set $S$.
  
The $\cO$-algebra structure that interchanges with itself endows the homotopy coefficient system of an $\cO$-algebra $X$ with natural transformations 
\[
\underline{\pi_{k}(X)}_{T}\to\underline{\pi_{k}(X)}
\]
for all admissible sets $T$ and which commute with restriction. If all sets are admissible, then this is equivalent to a Mackey functor structure on $\underline{\pi_{k}(X)}$~\cite{HillHopkins}. 
\end{remark}

\begin{remark}
One of the classical ways to package the data of a Mackey functor is
via additive functors from the Burnside category of spans of finite
$G$-sets into some other category.  There is an ``incomplete'' version
of these that can be used in our context.  The appropriate notion of a
``span'' for our incomplete Mackey functors is an isomorphism class of
a pair of maps $S\leftarrow U\rightarrow T$, where $U\to T$ is a
pull-back of a map between admissible sets.  These objects forms a
subcategory of the Burnside category. A full treatment of this
approach also engages with the issues from
Remark~\ref{rem:OperadsonGSets} of indexing our operads on finite
$G$-sets rather than on natural numbers.  We intend to return to this
issue in a subsequent paper.
\end{remark}

Having seen that the homotopy groups of an $\cO$-algebra in $G$-spaces
have transfers analogous to those possessed by the homotopy groups of
genuine spectra, we restrict attention to $\cO=\cK(U)$ or $\cO=\cD(U)$
for a universe $U$ and show that we are in fact constructing the usual
transfer. Recall that an equivariant $\aO$-algebra $X$ is
``group-like'' if $\pi_{0}(X^{H})$ is an abelian group for all
$H\subset G$.  We have the following delooping result:

\begin{proposition}[\cite{costenoblewaner}]
If $X$ is a group-like $\cK(U)$-algebra or $\cD(U)$-algebra then there
is an equivariant spectrum $\mathfrak X$ indexed on $U$ for which $X$
is the zero space. Similarly, a map of $\cK(U)$-algebras $X\to Y$
deloops to a map $\mathfrak X\to\mathfrak Y$ of spectra indexed on
$U$.
\end{proposition}

We can now deloop any of our structure maps since
Corollary~\ref{cor:TransfersAreLoopMaps} implies that they are
infinite loop maps.

\begin{corollary}
Fix some universe $U$, and let $H/K$ be an admissible $H$-set for
$\cK(U)$.  If $X$ is a grouplike $\cK(U)$-algebra, then we have a map
of spectra indexed by $U$:  
\[
F_{K}(H,\mathfrak X)\to i_{H}^{\ast}\mathfrak X,
\]
where $\mathfrak X$ is the spectrum whose zero space is $X$, and where
$F_{K}(H,\mathfrak X)$ is the coinduced spectrum. Moreover, the
homotopy class is unique. 
\end{corollary}

In this context, we see another interpretation of
Theorem~\ref{thm:ExistenceofNorms} (iii). The relevant spaces in the
operad $\cO$ parameterize the homotopies making the diagrams 
\[
\xymatrix{
{i_{K}^{\ast}N^{S}\mathfrak X\simeq N^{i_{K}^{\ast}S}i_{K}^{\ast} \mathfrak X}\ar@{<->}[rr] \ar[dr] & & {N^{i_{K}^{\ast}T}i_{K}^{\ast}\mathfrak X\simeq i_{K}^{\ast}N^{T}\mathfrak X}\ar[dl] \\
{} & {i_{K}^{\ast}\mathfrak X.} & {}}\]
commute. This is again an incarnation of the double-coset formula. 

When $\cO$ is $\cK(U)$ for some universe $U$, then these transfers
recover the classical transfers.

\begin{proposition}
If $X$ is a group-like $\cK(U)$-algebra, then the operadic transfer 
map associated to an admissible set $G/H$ gives rise to the
ordinary transfer. 
\end{proposition}

\begin{proof}
This identification essentially follows from the definition of the
action of the little disks operad on $\Omega^V S^V$.  Due to the
problems with suspension in the context of the little disks operad, we
will have to shift between $\cK(U)$ and $\cD(V)$ in the following
argument.

First, observe that if $G/H$ is an admissible $G$-set for $\cK(U)$,
then it is an admissible $G$-set for $\cD(U)$ and so for some finite
dimensional subspace $V \subset U$, we have a $G$-equivariant embedding  
\[
G/H\times D(V)\hookrightarrow D(V).
\]
For a particular subspace $V$, these choices can be inequivalent, but
letting the dimension grow yields our contractible space of maps 
\[
G/H\times D(U)\hookrightarrow D(U).
\]
Thus in the limit, any choices we made becomes equivalent, and we can
restrict attention to some finite dimensional $V$ and the $V$-fold
loops.

Since $X$ is a $\cK(U)$-space, delooping~\cite{costenoblewaner}
implies that $X \htp \Omega^V Y$ as a $\cK(U)$-space for some $Y$.
Changing operads, we can regard $X$ as having a $\cD(V)$ action which
is compatible with the $\cK(U)$ action.  Any embedding of the form
$G/H\times D(V)\hookrightarrow D(V)$ induces a Pontryagin-Thom map
\[
S^{V}\to G/H_{+}\wedge S^{V}.
\]
Taking maps out of this produces a map of algebras
\[
F_H(G_+,i_H^{\ast}\Omega^{V}Y)\cong F(G/H_+ \sma
S^{V},Y)\to \Omega^{V}Y,
\]
which in this case manifestly represents the same homotopy class as
the map constructed in Theorem~\ref{thm:NormsasTransfers}; the
Pontryagin-Thom collapse yields precisely the operadic structure map
in this case.  But of course this collapse is also the same as the
classical construction of the transfer map~\cite{Adams}.
\end{proof}

\begin{remark}
One can also deduce the preceding comparison of transfers from the
fact the description of the transfer as the composite of the inverse
of the Wirthmuller isomorphism and the action map $G \sma_H X \to
X$~\cite[4.15]{Schwede}.  Specifically, the result follows from this
characterization along with the fact that the delooping of the
operadic multiplication of a group-like $\aO$-space produces the fold
map of $G$-spectra. 
\end{remark}

\subsection{$\Ninfty$-ring spectra and the norm}

We now study the case of $\Ninfty$ algebras in orthogonal $G$-spectra.
The arguments are essentially the same as in the preceding subsection,
but the interpretation is different.  The proof of
the following is identical to the proof of
Theorem~\ref{thm:TransfersinSpaces} and
Proposition~\ref{prop:DoubleCoset}, so we omit it.

\begin{theorem}
If $R$ is an algebra over an $\Ninfty$ operad $\cO$, then
\[
\underline{\pi_{0}(R)}
\]
is a commutative Green functor.
 
If the $\cO$ action interchanges with itself, then for any admissible
$H$-set $H/K$ we have a ``norm map'' 
\[
\underline{\pi_{0}(R)}(G/K)\xrightarrow{n_{K}^{H}}\underline{\pi_{0}(R)}(G/H)
\]
which is a homomorphism of commutative multiplicative monoids. 

The maps $n_{K}^{H}$ satisfy the multiplicative version of the Mackey double-coset formula.
\end{theorem}

Thus just as the homotopy groups of algebras in spaces over the Steiner operad on an incomplete universe gave incomplete Mackey functors with only some transfers, the zeroth homotopy group of an algebra in spectra over the linear isometries operad on an incomplete universe gives incomplete Tambara functors with only some norms.

\appendix

\section{The homotopy theory of algebras over $\Ninfty$ operads in
$\Sp_{G}$}\label{sec:monadicalgebras}

In this section, we quickly present some technical results about the
abstract homotopy theory of categories of algebras over
$\Ninfty$ operads.

\subsection{Model structure and comparison results}

Given an $\Ninfty$ operad $\aO$, there is an associated monad $\bO$ on
$\Sp_{G}$ formed in the usual fashion: for an object $X$ in $\Sp_{G}$, the
free $\aO$-algebra can be described as
\[
\bO X = \bigvee_{n} \aO(n)_{+} \sma_{\Sigma_n} X^{\sma n}.
\]
The category of $\aO$-algebras is the category $\Sp_{G}[\bO]$ of algebras
over the monad $\bO$.  For our model category results, we require a
mild hypothesis on the spaces in the operad (we could equivalently
assume that the operads arise as the geometric realization of
simplicial operads).  We believe that in fact this sort of hypothesis
is unnecessary, but we do not study that issue here.

\begin{proposition}
Let $\aO$ be an $\Ninfty$ operad for which each space $\cO_n$ is of
the homotopy type of a $G \times \Sigma_n$-CW complex.  Then the
category of $\aO$-algebras has a model structure in which the weak
equivalences and fibrations are lifted from $\Sp_{G}$.
\end{proposition}

\begin{proof}
We use the criteria of~\cite[5.13]{MMSS}, which gives conditions for a
monad on a topological model category to generate lifted model
structures on the associated category of algebras.  As observed in the
argument for~\cite[B.130]{HHR}, the nontrivial aspect of verifying
these criteria is showing that given a generating acyclic cofibration
$A \to B$ and a map $\bO A \to X$, the map $X \to Y$ in the pushout
square 
\[
\xymatrix{
\bO A \ar[r] \ar[d] & \bO B \ar[d] \\
X \ar[r] & Y\\
}
\] 
is a weak equivalence.  It is easy to see that $\bO A \to \bO B$ is an
$h$-cofibration (i.e., it satisfies the homotopy extension property)
and so it suffices to show that $\bO A \to \bO B$ is a weak
equivalence.  By our hypotheses on $\aO$, the proof
of~\cite[B.115]{HHR} applies here to establish the analogue
of~\cite[B.113]{HHR}, which yields the result.
\end{proof}

\begin{warning}
The verification that $\bO$ takes the acyclic cofibrations to weak
equivalences is not trivial and can fail in other seemingly similar
situations.  For example, if we localize the category of orthogonal
spectra at the $E\tilde{F}$-equivalences, the free commutative algebra
monad does not preserve equivalences and so the construction of the
model structure on commutative ring objects fails.  This subtlety is
closely related to the localization phenomena discussed
in~\cite{HillHopkins}.
\end{warning}

Associated to a map $f \colon \cO \to \cO'$ of operads is an adjoint
pair 
\[
\xymatrix{
{f_{!}\colon \Sp_{G}[\cO]} \ar@<1ex>[r] & {\Sp_{G}[\cO']\colon f^{\ast}} \ar@<1ex>[l],
}
\]
where $f^*$ is the pullback and $f_!$ is the coequalizer
\[
\xymatrix{
{\bO' \bO X} \ar@<1ex>[r] \ar@<-1ex>[r]  & {\bO' X},
}
\]
where one map is the action map on $X$ and the other is the composite
of $f$ and the multiplication on $\bO'$.  In the standard model
structures on $\Sp_{G}[\cO]$ and $\Sp_{G}[\cO']$, it is clear that this pair
forms a Quillen adjunction since $f^*$ clearly preserves fibrations
and weak equivalences.

The following result justifies the notion of weak equivalence of
$\Ninfty$ operad.  The argument is a standard cellular induction
starting from the easy comparison on free algebras; e.g.,
see~\cite[3.14]{ABGHR}.

\begin{theorem}\label{thm:operadcompare}
Let $f \colon \cO \to \cO'$ be a weak equivalence of
$\Ninfty$ operads.  Assume that $\cO_1$ and $\cO'_1$ have
nondegenerate $G$-fixed basepoints and each $\cO_n$ and $\cO'_n$ are
of the homotopy type of $G \times \Sigma_{n}$ CW-complexes.  Then the
adjoint pair $(f_!, f^*)$ is a Quillen equivalence.
\end{theorem}

\subsection{Comparison to rigid commutative monoids}\label{sec:rigid}

In the category of orthogonal spectra, the symmetric
monoidal structure given by the smash product is constructed so that a
commutative monoid encodes the classical homotopy-coherence data of an
$E_\infty$ ring spectrum~\cite{MMSS}.  The key technical underpinning
of this comparison is the equivalence 
\begin{equation}\label{eq:sympow}
(E\Sigma_i)_+ \sma_{\Sigma_i} X^{\sma i} \to X^{\sma i} / \Sigma_i,
\end{equation}
for a positive cofibrant orthogonal spectrum $X$~\cite[15.5]{MMSS}.
Furthermore, since the category of orthogonal spectra is enriched in
spaces, we can consider $E_\infty$ objects in orthogonal spectra;
these have a homotopy theory equivalent to that of commutative monoids
and hence classical $E_\infty$ ring spectra~\cite[13.2]{Mayrant}.

The category of orthogonal $G$-spectra is also symmetric monoidal, and
we have the following analogue of equation~\eqref{eq:sympow} 
\begin{equation}\label{eq:eqsympow}
(E_G \Sigma_i)_+ \sma_{\Sigma_i} X^{\sma i} \to X^{\sma i} / \Sigma_i,
\end{equation}
for a positive cofibrant orthogonal $G$-spectrum~\cite[III.8.4]{MM}
(see also~\cite[B.117]{HHR}).  Once again, this implies that the
homotopy theory of commutative monoids is equivalent to the homotopy
theory of classical $E_\infty$ ring spectra (over
the linear isometries operad).  Moreover, we have the following
comparison between algebras over complete $\Ninfty$ operads and
commutative monoids in the category of orthogonal $G$-spectra, which
we follows from the same kind of inductive argument as 
Theorem~\ref{thm:operadcompare}, using the equivalence of
equation~\eqref{eq:sympow} to start the induction (i.e., to do the
comparison on the free algebras).

\begin{theorem}\label{thm:rigidcompare}
Let $X$ be an algebra in orthogonal $G$-spectra over a complete
$\Ninfty$ operad $\aO$.  Assume that $\aO$ has a nondegenerate
$G$-fixed basepoint and each $\aO_n$ has the homotopy type of a
$G \times \Sigma_{n}$ CW-complex.  Then there exists a commutative
monoid $\tilde{X}$ in orthogonal $G$-spectra such that
$X \htp \tilde{X}$ as algebras over $\aO$.  (Here we are using the
pullback along the terminal map from $\aO$ to the commutative operad
to give $\tilde{X}$ the structure of an $\aO$-algebra).  This
correspondence is functorial, and there is a zig-zag of equivalences
on Dwyer-Kan simplicial localizations between the category
$\Sp_G[\bP]$ of commutative monoids in $\Sp_G$ and the category
$\Sp_G[\bO]$ of $\aO$-algebras in $\Sp_G$.
\end{theorem}

In fact, using the same argument we can obtain a more general
comparison result on the category of orthogonal $G$-spectra indexed on
an incomplete universe $U$.  Specifically, there is a zig-zag of Dwyer-Kan
equivalence between algebras over the commutative operad and algebras
over any $\Ninfty$ operad weakly equivalent to the $G$-linear
isometries operad indexed on $U$.

\section{Operadic algebras and geometric fixed points}\label{sec:appfix}

One of the most important constructions in equivariant stable homotopy
theory is that of geometric fixed points for a normal subgroup
$N$ (e.g., see~\cite[\S V.4]{MM}).  We finish our general analysis of
$\cO$-algebras by describing the structure carried by their
$N$-geometric fixed points.  We let $\Phi^{N}(-)$ denote the point-set
$N$-geometric fixed point functor~\cite[\S V.4]{MM}. 

We first address the effect of fixed points on the operad and the
admissible sets. 

\begin{lemma}
Let $N$ be a normal subgroup of $G$, and let $\cO$ be an $\Ninfty$ operad. Then
\begin{enumerate}
\item $\cO^{N}$ is a $\Ninfty$ $G/N$-operad and
\item the admissible $H/N$-sets for $\cO^{N}$ are the admissible $H$ sets for $\cO$ which are fixed by $N$.
\end{enumerate}
\end{lemma}
\begin{proof}
It is obvious that $\cO^{N}$ still forms a $G$-operad, and it is also clear that there are no fixed points for the symmetric groups. Both parts of the lemma then rely on understanding the way families behave upon passage to fixed points by a normal subgroup. Let $\Gamma$ be a subgroup of $G/N\times\Sigma_{n}$. Then 
\[
(\cO_{n}^{N})^{\Gamma}=\cO_{n}^{\pi^{-1}(\Gamma)},
\]
where $\pi\colon G\to G/N$ is the canonical projection, is either
empty or contractible. Thus $\cO_{n}$ is in fact a universal space,
making $\cO^{N}$ an $\Ninfty$ $G/N$-operad. 

For the second part, we again use the above equality of fixed points. If $\Gamma_{T}$ corresponds to an admissible $H/N$-set $T$ for $\cO^{N}$, then the above equality shows that $\pi^{-1}(\Gamma_{T})$ corresponds to an admissible $H$-set for $\cO$. Since this contains $N\times\{1\}$, we see that this admissible $H$-set is simply $T$ again, now viewed as an $H$-set. Thus the admissible sets for $\cO^{N}$ are precisely the admissible sets for $\cO$ which are fixed by $N$.
\end{proof}

The $\Ninfty$ $G/N$-operad $\cO^{N}$ is also an $\Ninfty$ operad via
the quotient $G\to G/N$. Thus it is a sub $\Ninfty$ operad of $\cO$,
and by restriction of structure, any $\cO$-algebra $R$ is also a
$\cO^{N}$-algebra. This is the heart of the following theorem. 

\begin{theorem}\label{thm:GeomFP}
If $R$ is an $\cO$-algebra, then $\Phi^{N}(R)$ is an $\cO^{N}$-algebra.
\end{theorem}
\begin{proof}
Since $N$ acts trivially on $\cO^{N}$, the fact that $\Phi^{N}$ is lax
symmetric monoidal gives rise to a canonical composite 
\[
\cO^{N}_{n+}\wedge_{\Sigma_{n}}\Phi^{N}(R)^{n} \to \cO^{N}_{n+}\wedge_{\Sigma_{n}}\Phi^{N}(R^{n}) \to \Phi^{N}(\cO^{N}_{n+}\wedge_{\Sigma_{n}}R^{n}).
\]
All of our structure maps are then induced by $\Phi^{N}$ applied to
the structure maps for the $\cO^{N}$-algebra $R$. 
\end{proof}

\begin{corollary}\label{cor:Combinatorial}
If $\cO'$ is any sub $\Ninfty$ operad of $\cO$ on which $N$ acts trivially and $R$ is an $\cO$-algebra, then $\Phi^{N}(R)$ is an $\cO'$-algebra.
\end{corollary}

In particular, in the absolute worst case possible, we choose $\cO'$
to be the $G$-fixed subspace. The only admissible sets are those with trivial action (and this becomes an operad modeling a ``coherently homotopy commutative multiplication'' with no other structure). Then Corollary~\ref{cor:Combinatorial} shows that for any $\cO$-algebra $R$ and for any normal subgroup $N$, $\Phi^{N}(R)$ is a $\cO'$-algebra and in particular, has a coherently homotopy commutative multiplication.

\begin{remark}
The same statements are true for the actual fixed points, rather than
the geometric fixed points. The proofs also largely carry through
{\emph{mutatis mutandis}}. The only change is in the proof of
Theorem~\ref{thm:GeomFP}, in which the homeomorphism comparing
$\Phi^{N}(R)^{\wedge n}$ and $\Phi^{N}(R^{\wedge n})$ is replaced by 
a map 
\[
(R^{N})^{\wedge n}\to (R^{\wedge n})^{N}.
\]
See~\cite{westerland} for analysis of operads obtained in this fashion.
\end{remark}

\end{document}